\documentclass[11pt,a4paper]{article}

\usepackage{amssymb,amsmath,amsfonts}
\usepackage{times}
\usepackage{mathrsfs}
\usepackage{enumerate}

\usepackage[numbers,sort&compress]{natbib}
\usepackage{graphicx,titlesec,titletoc}
\usepackage{pdfpages}
\usepackage{authblk}
\usepackage{ulem}
\usepackage[bookmarks=true,colorlinks=true,urlcolor=blue,citecolor=red,linkcolor=blue,linktocpage,pdfpagelabels,bookmarksnumbered,bookmarksopen]{hyperref}
\usepackage{amsthm}
\newtheorem{thm}{Theorem}[section]
\newtheorem{remark}[thm]{Remark}
\newtheorem{defn}[thm]{Definition}
\newtheorem{lem}[thm]{Lemma}
\newtheorem{pro}[thm]{Proposition}

\def\R{\mathbb{R}}
\def\N{\mathbb{N}}

\def\supp{\mbox{supp}}
\def\dist{\mbox{dist}}

\def\norm#1{\|{#1}\|}

\numberwithin{equation}{section}
\usepackage{braket}
\usepackage[left=3.1cm,right=3.1cm,top=3.2cm,bottom=3.2cm]{geometry}

\title{\bfseries Segregated solutions for  nonlinear Schr\"odinger systems  with sublinear coupling terms}
\author{Qing Guo\thanks{guoqing0117@163.com}}

\affil[1]{\footnotesize College of Science, Minzu University of China, Beijing 100081, China}

\author{Chengxiang Zhang\thanks{zcx@bnu.edu.cn (corresponding author)}}
 \affil[2]{\footnotesize Laboratory of Mathematics and Complex Systems (Ministry of Education), School of Mathematical Sciences, 
 Beijing Normal University, Beijing 100875, P. R. China}
 \date{} 

\begin{document}
\maketitle

\begin{minipage}{14cm}
	\noindent{\bfseries Abstract:} 
In this paper, we employ an enhanced version of the Lyapunov--Schmidt reduction method to study a particular class of nonlinear Schrödinger systems featuring sublinear coupling terms. Under suitable assumptions, we establish the existence of infinitely many nonnegative, segregated solutions for the system
	\begin{equation*}
		\left\{\begin{aligned}-\Delta u+K_1(x)u&=\mu u^{p-1}+ (\sigma_1+1)\beta u^{\sigma_1}v^{\sigma_2+1}, &x\in\mathbb{R}^N&,
		\\
		-\Delta v+K_2(x)v&=\nu v^{p-1}+(\sigma_2+1)\beta u^{\sigma_1+1}v^{\sigma_2}, &x\in\mathbb{R}^N&,\end{aligned}\right.
	\end{equation*}
	where $N\geq 2$, $ p\in(2,2^*) $ with $ 2^* = \frac{2N}{N-2} $ denoting the critical Sobolev exponent if $ N \geq 3 $ (and $ 2^* = \infty $ when $ N = 2 $). The functions $ K_j(x) $, $ j = 1, 2 $, are radially symmetric potential functions, the exponents $ \sigma_j \in (0,1) $ correspond to sublinear coupling terms, $ \mu > 0 $ and $ \nu > 0 $ are given constants, and $ \beta \in \mathbb{R} $ acts as the coupling coefficient.
	
	\qquad The range of the exponents $\sigma_j$ introduces substantial challenges to classical reduction methods, primarily due to the nonsmoothness and sublinearity inherent in the coupling terms. To address these difficulties, we introduce a novel approach that recasts the reduction process as a fixed point problem defined on an appropriately constructed metric space. This space is formed by local minimizers of an associated outer boundary value problem and is furnished with crucial a priori estimates, which together enable us to verify the contraction mapping property.

	\qquad  Moreover, we identify a novel phenomenon in the sublinearly coupled regime: the constructed solutions $(u_\ell, v_\ell)$ exhibit a distinct “dead core” behavior, characterized by non-strict positivity. In particular, for $N = 2$, we show that the supports of the components separate as follows: for each sufficiently large integer $\ell$, there exist radii $0 < R_1 < R_2$, depending on $\ell$, such that $\supp u_\ell \subset B_{R_2}(0)$, $\supp v_\ell \subset \R^N\setminus B_{R_1}(0)$, and $u_\ell + v_\ell \to 0$ uniformly in the annular region $B_{R_2}(0) \setminus B_{R_1}(0)$ as $\ell \to \infty$.

	\qquad We believe that the framework developed here has broad applicability and can be used to tackle other problems involving similar nonsmooth nonlinearities.

\medskip

	
\end{minipage}
 % \date{}

\section{Introduction}

\subsection{Backgrounds and problems}

The coupled nonlinear Schr\"odinger system
\begin{equation}\label{eq00}
	\begin{cases}
	 -\mathrm{i}\partial_t\Phi_j=\Delta\Phi_j-V_j(x)\Phi_j+\mu_j|\Phi_j|^{p-2}\Phi_j+\sum_{i\neq j}\beta_{ij}|\Phi_i|^{\sigma_i+1}|\Phi_j|^{\sigma_j-1}\Phi_j,\\
	 \Phi_j=\Phi(x,t)\in\mathbb C,\  x\in\R^N,t>0, \ j=1,2
	\end{cases}
	\end{equation}
has applications in many physical problems, especially in nonlinear optics (see \cite{AA}). Here, $p\in(2,\frac{2N}{N-2})$,
$\sigma_j>0$, $\Phi_j$ represents the $j$th component of the beam in Kerr-like photo-refractive media, $\mu_j$ is a constant for self-focusing in the $j$th component of the beam, and the coupling constant $\beta_{ij}$ indicates the interaction between the $i$th and $j$th components of the beam. A positive $\beta$ implies an attractive interaction, causing the components of a vector solution to come together. Conversely, a negative $\beta$ indicates the repulsive interaction, causing the components to repel each other and forming the phase separations. The functions $V_1(x)$ and $V_2(x)$ denote the magnetic trapping potentials. For further clarification on the constants, one can refer to \cite{BSSS}.

	To seek solitary wave solutions of the system \eqref{eq00}, we introduce the ansatz $\Phi_j(x,t) = e^{-\mathrm{i}\lambda_jt}u_j(x)$, where $\lambda_j \in \mathbb{R}$ and $u_j \in \mathbb{R}$. This transforms the system \eqref{eq00} into a set of semilinear elliptic equations, expressed as:
	\begin{equation}\label{eq0}
		\left\{\begin{aligned}-\Delta u+K_1(x)u&=\mu |u|^{p-2}u+(\sigma_1+1)\beta |v|^{\sigma_2+1}|u|^{\sigma_1-1}u,\quad x\in\mathbb{R}^N,
		\\
		-\Delta v+K_2(x)v&=\nu |v|^{p-2}v+(\sigma_2+1)\beta |u|^{\sigma_1+1}|v|^{\sigma_2-1}v,\quad x\in\mathbb{R}^N,\end{aligned}\right.
	\end{equation}
	where $K_j(x)=V_j(x)+\lambda_j, j=1,2$  represent continuous positive functions, while $\mu>0$, $\nu>0$, and $\beta\in\mathbb{R}$ denote the coupling constant. These equations correspond to the time-independent vector Gross-Pitaevskii or Hartree-Fock equations for the condensate wave function $u$ and $v$.

Over the last two decades, extensive research has focused on systems of Schrödinger equations, particularly in two- and three-dimensional spaces featuring cubic nonlinearity with subcritical growth and the absence of potentials:
\begin{align}\label{eqlambda}
	\begin{split}\begin{cases}
	-\Delta u+\lambda_1u=\mu |u|^{2}u+\beta |v|^{2}u,\quad x\in\mathbb{R}^N,
		\\
		-\Delta v+\lambda_2v=\nu |v|^{2}v+\beta |u|^{2}v,\quad x\in\mathbb{R}^N.\end{cases}
	\end{split}
	\end{align}
	 This research has yielded significant advancements in understanding the existence, symmetry, uniqueness, and asymptotic behavior of ground state solutions. Additionally, numerous studies have delved into exploring positive solutions, as documented in various works \cite{3,13,15,21,LW05,MMP,31,32}.
Trapping potentials, as discussed in \cite{lw}, exert a profound influence on the behavior of spikes. When $\beta<0$, spikes become separated and ensnared at the minimum points of the potentials. Conversely, when $\beta>0$, a competition arises between the attractions of spikes and the trapping potential wells. If the potential wells exert a sufficiently strong attraction, the spikes disperse, while if the attraction of the potential wells is inadequate, the spikes come together. %These observations underscore the significant impact of interspecies scattering length and potentials on the spatial distribution of spikes, rendering the interaction among them complex and challenging to decipher.
For further insights into this phenomenon, interested readers can refer to \cite{ambrosetti-colorado-ruiz,BC13,GLWZ19-1,GLWZ19-2,GLW,GS,Guo1,guo1,guo2,RL14}. 

\smallskip

Unlike single equations, system \eqref{eq0} can accommodate semi-trivial solutions where one or more components vanish identically. Additionally, synchronized solutions, where the components are multiples of each other, can also exist. However, the focus always lies in identifying fully nontrivial (and non-synchronized) solutions, where none of the components are identically zero. 
 Segregation involves investigating the characteristics of solutions that exhibit disjoint supports.
In particular, concerning the solutions with concentration properties, synchronized solutions refer to cases where the concentration points of the two components coincide. As concentration takes place, they converge towards a synchronized type of vector solution of the limiting elliptic system. On the other hand, the components of the segregated solutions  gradually separate, ultimately converging to solutions of different single equations respectively.
This segregation phenomenon has been experimentally observed in two-component Bose-Einstein condensates (BECs), which is a significant area of study in experimental physics \cite{HMEWC}.
For a detailed description of the physics involved, refer to \cite{KTU}.

\smallskip

In \cite{PW13}, Peng and Wang utilized the finite-dimensional reduction method to generate an unbounded sequence of non-radial positive vector solutions  for the Schr\"odinger system \eqref{eq0} with $p=4$, $\sigma_1=\sigma_2=1$ where they elucidated the construction of an unbounded sequence of non-radial positive vector solutions of synchronized and segregated type. There have been many works that extend the conditions on the potential functions such as \cite{wz,aw,awy,zheng,pistoiavaira}. However, these extensions are mostly focused on the sublinear coupling exponents $\sigma_j\geq 1$, $j=1,2$.

\smallskip

 Existing research largely focuses on systems in dimensions $N=2,3$ with $p=4$ and $\sigma_j \geq 1$. Results for sublinear couplings $\sigma_j \in (0,1)$ remain sparse.  For instance, \cite{BL, CZ} established the existence of ground states via concentration-compactness arguments, while \cite{Stuart, BBT} used symmetrization to prove existence of nontrivial minimal-action solutions. Maia et al. \cite{MMP} provided sufficient conditions on $\beta$ to guarantee positive ground states by comparing Morse indices of semi-trivial solutions $(u,0)$, $(0,v)$ to coupled pairs $(u,v)$.

 The sublinear coupling regime $\sigma_j < 1$ introduces fundamental challenges. Segregated solutions become essential as classical techniques fail to extend directly, as the linearized operator involves negative exponents $\sigma_j - 1$. These exponents introduce significant challenges for the reduction method, inevitably producing singularities during direct construction attempts.
 It is known that the Lyapunov-Schmidt reduction method involves constructing concentrated solutions within the perturbation framework. For each component, a solution is considered as a small perturbation of some explicit function that solves the limit system, chosen as a suitable approximation. This process becomes particularly challenging when using the implicit function theorem to establish the continuous smoothness of solutions with respect to parameters, as such singularities are unavoidable when  $\sigma_j<1$.
On the other hand, the coupling terms exhibit a degree of nonlinearity less than 1, making the direct application of the contraction mapping theorem unfeasible. Indeed, {\bf sublinearly coupled systems have long been recognized as challenging problems and existence results on constructing concentrated solutions using reduction methods are scarce.} This gap represents  open problems, as underscored in various references such as \cite{pistoia-coron} and \cite{glwjde}, which explicitly delineate the difficulties posed by sublinear coupling exponents and acknowledge that current methodologies are insufficient to address such issues.
\smallskip

To address this gap, we extend \cite{PW13} to $\sigma_j \in (0,1)$ through a novel approach to segregated solutions. Direct application of reduction methods is impeded by solution decay in outer regions and nonsmooth coupling nonlinearities, particularly when implementing fixed-point arguments in complement spaces during finite-dimensional reduction. Our strategy—generalizing and simplifying \cite{WZZ}—resolves these issues by partitioning space into inner regions (near concentration points) and outer regions (where solution decay occurs). This partitioning is motivated by the observation that sublinearity-induced singularities manifest where solutions vanish.

Key to our approach is the construction of a tailored functional space with controlled decay properties. Using a tail minimization technique inspired by variational gluing methods \cite{Coti92,BT14}, we define a complete metric space where functions satisfy sharp a priori decay estimates. Crucially, we establish decay rates not only for approximate solutions but also for:
(i) Deviations from limiting profiles ;
(ii) Differences between distinct functions in the space.

Unlike earlier weighted space approaches, this framework enables rigorous contraction mapping arguments. By restricting the fixed-point problem to this metric space and leveraging our decay estimates, we successfully reduce the system to a finite-dimensional setting, circumventing the singularity and decay barriers that obstruct classical methods.

\smallskip

We believe this approach offers broad applicability to a variety of related problems and holds promise for addressing several open questions, making it a valuable avenue for further research and exploration.

%%%%%%%前面这些段落有一些放的定理后面的注释里面

\medskip

\subsection{Main results}

In this paper, we consider 
  the  nonlinear Schr\"odinger system 
  	\begin{equation}\label{eq0'}
		\left\{\begin{aligned}-\Delta u+K_1(x)u&=\mu |u|^{p-2}u+\beta \partial_1 G(u, v),\quad x\in\mathbb{R}^N,
		\\
		-\Delta v+K_2(x)v&=\nu |v|^{p-2}v+\beta \partial_2 G(u, v),\quad x\in\mathbb{R}^N,\end{aligned}\right.
	\end{equation}
  with
  $N\geq2$, $p\in (2,2^*)$ and    the potential $K_1, K_2 \geq 1$  satisfy
\begin{itemize}
\item[(K)]   There are constants $a_1, a_2 >0$, $m>1$, and $\theta >0$ such that as $r\to+\infty$,
\[
K_i(r)=1+\frac {a_i}{r^m}+O\left(\frac1{r^{m+\theta}}\right), \quad i=1,2.
\]
\end{itemize}
Here, $ \partial_i G $ denotes the $ i $-th partial derivative of $ G $. The function $ G(u, v) $ satisfies the following assumptions:
\begin{itemize}
	\item[(G1)] $G\in C^1(\R^2)$,  $G(s_1, s_2)=G(|s_1|, |s_2|)$ for any $(s_1,s_2)\in \R^2$, and $G(s_1,s_2)=0$ if $s_1=0$ or $s_2=0$;
	% \item[(G2)] $\limsup_{|s_1|+|s_2|\to\infty}\frac{G(s_1, s_2)}{|s_1|^{2^*} +|s_2|^{2^*}}<+\infty$.
	\item[(G2)] For $i=1,2$,  $\partial_i G(s_1, s_2)$ is increasing in $s_i$ and is $C^1$ if $s_i\neq 0$.
	\item[(G3)]     
	There exists $\sigma\in(0,1)$ such that  
	 \[ |s_is_j\partial_{ij} G(s_1,s_2)|=o(|s_1 s_2|^{\sigma+1})\quad
	\text{as}\quad |s_1 s_2|\to 0, \quad\text{for } i=1,2, j=1,2.\]
	Here $\partial_{ij} G(s_1,s_2)=\partial_i\partial_j G(s_1,s_2)$.
\end{itemize}
Our main results are the following.
\begin{thm}\label{thm1}
    Suppose that    conditions {\rm (K)}, and {\rm (G1)}--{\rm (G3)} hold. Then, for $\beta < 0$, problem~\eqref{eq0'} admits infinitely many non-radial segregated solutions $(u_\ell, v_\ell)$ whose energy can be made arbitrarily large. 
	% Moreover, we have
    % $$
    % \liminf_{\ell \to +\infty} \max u_\ell > 0 \quad \text{and} \quad \liminf_{\ell \to +\infty} \max v_\ell > 0.
    % $$
\end{thm}

	The segregated nature of these solutions is elucidated in Theorems \ref{thm1'} later in this paper. In essence, segregated solutions can be viewed as small perturbations of $(U_r,V_\rho)$, where $U_r$ and $V_\rho$ are composed of translated $W_\mu$ and $W_\nu$ functions, respectively, positioned at the vertices of a large $\ell$-polygon ($2\ell$-polygon, respectively) with distinct radii $r$ and $\rho$. Here, $W_\mu$ and $W_\nu$ denote the unique positive radial solutions of $-\Delta w+w= \mu w^{p-1}$ and $-\Delta w+w= \nu w^{p-1}$, respectively. In simpler terms, segregated solutions exhibit numerous bumps near infinity, with the bumps for $u$ and $v$ distributed on different circles.
	To illustrate this phenomenon, we first
 introduce some notations. 
\medskip 

Define $\mathbb H$ to be the Hilbert space 
$ H^1(\R^N)\times H^1(\R^N)$   
with the 
inner product 
\[
\langle (u_1, v_1), (u_2, v_2)\rangle =
 \langle u_1, u_2\rangle +\langle v_1, v_2 \rangle
\]
and the norm
$$
\|(u,v)\|=\sqrt{ \|u\|^2+\|v\|^2},
$$
where $\langle\cdot, \cdot \rangle$ and $\|\cdot\|$ denote the standard inner product and norm in $H^1(\R^N)$.

Set $x=(x',x''),x'\in\mathbb{R}^2,x''\in\mathbb{R}^{N-2}$.
Let
\begin{align}\label{points}
	\begin{split}
x^j&=\left(r\cos\frac{2(j-1)\pi}\ell,r\sin\frac{2(j-1)\pi}\ell,0\right):=(x^{\prime j},0),\quad j=1,\ldots,\ell,\\
y^j&=\left(\rho\cos\frac{(j-1)\pi}{\ell},\rho\sin\frac{(j-1)\pi}{\ell},0\right):= ({y'}^j,0 ),\quad  j=1,\ldots,2\ell,\end{split}
\end{align}
 where $0$ is the zero in $\R^{N-2}$, $r \in[r_0\ell\ln\ell,r_1\ell\ln\ell], \rho\in[2r_0\ell\ln\ell, 2r_1\ell\ln\ell]$ for some $r_1>r_0>0.$

Define 
\begin{align*}
	H_{ s}& =\Bigg\{u\in H^1(\mathbb{R}^N)\ \Big|\ u\text{ is even in }x_h,h= 2,\dots,N,  \\
	& u(r\cos\theta,r\sin\theta,x_3)=u\left(r\cos\left(\theta+\frac{2\pi j}\ell\right),r\sin\left(\theta+\frac{2\pi j}\ell\right),x''\right),\\
	& u(r\cos\theta,r\sin\theta,x_3)=u\left(r\cos\left(\frac{ 2\pi }\ell-\theta\right),r\sin\left( \frac{2\pi }\ell-\theta\right),x''\right)\Bigg\},
	\end{align*}
	and 
	\begin{align*}
		\mathbb H_{s}  =\Set{ (u, v)\in \mathbb H | u\in H_{ s}, v\in H_{ s} }.
	\end{align*}
 %\begin{remark}Note that the symmetry requirement in the definition of the space $H_{s} $ here is to ensure the coercivity of the linearized operator and to easily verify that the approximate solutions $ U_r(x)$ and $ U_\rho(x)$  given below exist within this symmetric space.\end{remark}

Let $W  $ be the unique solution of the following problem
 \[\begin{cases}-\Delta w+w= w^{p-1},&w>0\quad \text{ in }\mathbb{R}^N,\\w(0)=\max_{x\in\mathbb{R}^N}w(x),&w(x)\in H^1(\mathbb{R}^N).\end{cases}
 \]
It is well-known that $W$ is non-degenerate and $ W(x)=W(|x|),W'<0. $
Let $U=\mu^\frac{1}{2-p}W$, $V=\nu^\frac{1}{2-p}W$. We set 
\[ U_r(x)=\sum_{j=1}^{\ell}U_{ x^j}(x),\quad V_\rho(x)=\sum_{j=1}^{2\ell}V_{ y^j}(x),
 \]
where $U_{ \xi}(x)=U (x-\xi)$, $V_\xi(x)=V(x-\xi)$ for  $\xi\in \R^N$.

% \begin{pro}\label{energyexpansion} There is $\sigma>0$ such that
% 	\[I(U_r, V_\rho)=\ell(A +\frac{B_1}{r^m}+\frac{B_2}{\rho^m}-C_1(\frac{\ell}{r})^\frac{N-1}{2}e^{-\frac{ 2\pi r}{\ell}}-C_2(\frac{\ell}{\rho})^\frac{N-1}{2}e^{-\frac{ \pi \rho}{\ell}}+O(\frac{1}{r^{m+\sigma}}+\frac1{\rho^{m+\sigma}})).\]
% \end{pro}
In this paper, we assume 
 \[
	r \in\mathbb{D}_{1}:=\left[  \frac{m}{2\pi} \ell\ln\frac{\ell}2,\frac{m}{2\pi} \ell\ln(2\ell)\right],\ 	\rho\in \mathbb{D}_{2}:=\left[ \frac{m}{ \pi} \ell\ln\frac{\ell}2, \frac{m}{ \pi}\ell\ln(2\ell)\right].
 \] 
\medskip

We verify Theorem \ref{thm1} by proving the following result.
 \begin{thm}\label{thm1'}
	Under the assumptions of Theorem \ref{thm1}, there exists a sufficiently large constant $\ell_0 > 0$ such that for all $\ell \geq \ell_0$, the equation \eqref{eq0'} admits a solution $(u_\ell, v_\ell)$ of the form
\[
u_\ell = U_{r_\ell} + \varphi_\ell, \quad v_\ell = U_{\rho_\ell} + \psi_\ell,
\]
where $(\varphi_\ell, \psi_\ell) \in \mathbb{H}_s$, $r_\ell \in \mathbb{D}_1$, and $\rho_\ell \in \mathbb{D}_2$. Moreover, as $\ell \to \infty$, 
\[
\|(\varphi_\ell, \psi_\ell)\| \to 0.
\]
\end{thm}

\begin{remark}
The sublinear coupling exponent does not pose a fundamental challenge when constructing synchronized solutions. This is due to the fact that the concentration points of the two components of synchronized solutions coincide, eliminating singularities in the linearized operator. 
By making slight adjustments to the proof of \cite{PW13}, we can directly establish the reduction framework to obtain the synchronized solutions without the intricate decomposition process.

We require $\beta < 0$ to ensure convexity of the corresponding functional in the outer region, thereby guaranteeing the existence of a minimal energy solution for the outer problem. 
This condition represents a specific technical requirement for sublinear coupling exponents. In contrast, \cite{PW13} focuses solely on cubic polynomials, enabling the construction of segregated solutions with a positive coupling coefficient $\beta$ (althouth sufficiently small).
\end{remark}
\begin{remark}
	%The key to the distribution of the concentration points \eqref{points} designed in this paper is to achieve the segregation of the two components of the system. In fact, 
	There are other different ways to segregate the couples. For example, our method can easily handle the following scenario: one component is concentrated at the origin, while the other  is concentrated at the vertices of a regular polygon at infinity, which actually will result in a simpler proof.
		\end{remark}

The existence results in Theorem \ref{thm1} and Theorem \ref{thm1'} apply to system \eqref{eq0} for $\sigma_j>0$, $j=1,2$. However, the sublinear case $\sigma_j \in (0,1)$ exhibits fundamentally different behavior.

A key distinction in the sublinear regime is that segregated solutions may fail to remain strictly positive. Notably, a phenomenon known as the "dead core" (see \cite{Pucciserrin}) emerges in these segregated solutions. As stated in Theorem \ref{deadcore}, we can show that each component of the segregated solutions develops a large dead core region.

To establish this result, we require an additional assumption ensuring the nonlinearity is genuinely sublinear:
\begin{itemize}
	\item[(G4)]     
	There exists $\sigma'\in(0,1)$,   such that  
	 \[ \liminf_{|s_1 s_2|\to 0} \frac{s_i \partial_{i } G(s_1,s_2)}{|s_1 s_2|^{\sigma'+1}} >0\quad
	  \quad\text{for } i=1,2.\]
\end{itemize}

\begin{thm}\label{deadcore}
	  Suppose that    conditions {\rm (K)}, and {\rm (G1)}--{\rm (G4)} hold. For any sufficiently large $\ell$,  the solution
	$(u_\ell, v_\ell)$ obtained in Theorem \ref{thm1'} satisfies
	\[
	u_\ell=0\mbox{ in } \bigcup_{j=1}^{2\ell} B_{\frac{m+\tau}2\ln \ell}(y^j), \quad   v_\ell=0 \mbox{ in } \bigcup_{j=1}^\ell B_{\frac{m+\tau}2\ln \ell}(x^j).
	\]
	Specifically, when $N=2$,
	\[
	u_\ell(x)=0 \mbox{ if }|x|\geq |y^1|-\sqrt{m\tau/2}\ln\ell,\quad v_\ell(x)=0 \mbox{ if }|x|\leq  |x^1|+\sqrt{m\tau/2}\ln\ell.
	\]
\end{thm}
Complete segregation cannot occur; that is, the supports of $u$ and $v$ cannot be disjoint. Indeed, if the supports were disjoint, there would be no interaction between the components, implying that each function would independently solve a single equation. However, this contradicts the fact that both solutions have compact support.
A stronger argument confirms this conclusion through the strong unique continuation property established in \cite{Clapp25}. This result asserts that if $u$ and $v$ simultaneously vanish to infinite order at some point, then both functions must be identically zero. 

\begin{remark}
 	In this present work, we are investigating a Sobolev subcritical problem. Actually, our method remains applicable to elliptic systems with Sobolev critical exponents in higher dimensions ($N\geq5$), which also involve sublinearly coupled nonlinearities.  
\end{remark}

	\medskip

%This paper is organized as follows. In Section 2, we give some notations and some repeatedly used estimates. In Section 3, we establish and solve a minimization problem for the outer region and obtain important estimates for the minimization operator. Then we  modify the reduction framework and reduce the problem to the finite-dimension. In Section 5, we give the proof for the main Theorems.

\subsection{Organization} The structure of the paper is as follows. In Section 2, some modifications and basic estimates are introduced for the proof of Theorem \ref{thm1'}. In Section 3, we construct a minimization problem for the outer region and prove various a priori estimates for the solution of the exterior problem, which are crucial for the smooth progress of the subsequent reduction.
 Section 4 is the core of this work, where we address the challenge of sublinear coupling  by using the result from Section 3.
 We prove the existence of solutions using the fixed point theorem in the orthogonal complement of the kernel space, reducing it to a finite-dimensional problem. The final section completes the proof of the theorem by solving the finite-dimensional problem.

\section{Preliminaries}\label{sec2}

 Denote 
	\begin{align*}  \Omega_j&=\Set{z=(z',z'')\in\mathbb{R}^N | \left\langle\frac{z'}{|z'|},\frac{x'^j}{|x'^j|}\right\rangle\geqq\cos\frac{\pi}{\ell} }, j=1,\ldots,\ell,\\
		\widetilde\Omega_j&=\Set{z=(z',z'')\in\mathbb{R}^N | \left\langle\frac{z'}{|z'|},\frac{y'^j}{|y'^j|}\right\rangle\geqq\cos\frac{\pi}{2\ell} }, j=1,\ldots,2\ell.
	 \end{align*}	

	 Let $\chi_0\in C_0^1(\R^N)$ be a fixed truncation function such that $\chi_0=1$ in $B_{1}(0)$ and $\chi_0=0$ in $\R^N\setminus B_2(0)$.
		Let 
	\begin{align*}
X(r) &= \sum_{j=1}^\ell X_j(r), \quad \text{where} \quad X_j(r)(x) = \chi_0(x - x^j) \frac{\partial U_{x^j}}{\partial r}(x), \quad j = 1, \ldots, \ell, \\
Y(\rho) &= \sum_{j=1}^{2\ell} Y_j(\rho), \quad \text{where} \quad Y_j(\rho)(x) = \chi_0(x - y^j) \frac{\partial V_{y^j}}{\partial \rho}(x), \quad j = 1, \ldots, 2\ell.
\end{align*}

Define
\[
\mathbb{E}= \mathbb E_{r,\rho} = E_r \times E_\rho,
\]
where
\begin{align*}
E_r &= \Set{ u \in H_s | \int_{\R^N} {X}(r) u=0} = \Set{ u \in H_s |  \int_{\R^N}X_j(r)u=0, \quad j = 1, 2, \ldots, \ell  }, \\
E_\rho &= \Set{ v \in H_s | \int_{\R^N} Y(\rho) v=	0} =\Set{v\in H_s | \int_{\R^N}Y_j(\rho)v=0, \quad j = 1, 2, \ldots, 2\ell  }.
\end{align*}
		It is clear that $\mathbb{E}$  is a closed subspace of the Hilbert space $\mathbb H_s$.

\subsection{Modifications}   

Let 
\begin{align}\label{I}
	 I\left(u,v\right) =&\frac{1}{2}\int_{\mathbb{R}^N}\big(|\nabla u|^2+K_1(|x|)u^2+|\nabla v|^2+  K_2(|x|)v^2\big)  \nonumber \\
	 &-\frac1{ p}\int_{\mathbb{R}^N}\bigl(\mu|u|^{p}+\nu|v|^{p}\bigr)-\beta\int_{\mathbb{R}^N}G(u,v),\quad (u,v)\in\mathbb  H. 
	\end{align}
	Then, formally, the critical points of $ I $ are solutions to \eqref{eq0'}.
\medskip

Note that    the functional $I(u,v)$ in \eqref{I} is possibility not of class $C^2$  on $\mathbb H_s$. 
To address this issue, we proceed with some modification.

First by (G1),
we know 
\begin{equation}\label{eqnabla}\nabla G(s_1, s_2)=(0,0)\quad  \text{if}\  s_1=0 \text{ or } s_2=0.\end{equation}
Then by (G2), we have 
\[\partial_i G(s_1, s_2) \geq 0\  \text{if}\  s_i\geq 0 \mbox{ for } i=1,2, \mbox{ and hence, } G\geq 0 \mbox{ in } \R^2.
\]
By (G3), for any $\alpha>0$, there is $s_\alpha>0$ such that for each pair $(i,j)$,
\begin{equation}\label{sij}
|s_is_j\partial_{ij} G(s_1, s_2)|\leq |s_1 s_2|^{\sigma+1} \text{ for } (s_1,s_2)\in [-\alpha, \alpha]\times [-s_\alpha, s_\alpha] \cup [-s_\alpha, s_\alpha]\times [-\alpha, \alpha].
\end{equation}
 
  By $s_i\partial_i G(s_1,s_2)\geq0$ and $\beta\leq 0$, a solution $(u,v)$ of 
\eqref{eq0'} satisfies
\begin{equation}\label{eqsub}
		\left\{\begin{aligned}-\Delta |u|+K_1(x)|u|&\leq \mu |u|^{p-1},\quad x\in\mathbb{R}^N,
		\\
		-\Delta |v|+K_2(x)|v|&\leq\nu |v|^{p-1},\quad x\in\mathbb{R}^N.\end{aligned}\right.
	\end{equation}
Therefore, we can get that for some $\alpha_0>4\max\{U(0),V(0)\}$ independent of $G$ such that 
any solution $(u,v)\in \mathbb H$ satisfies \eqref{eq0'} with $\|(u,v)-(U_\ell, V_\ell)\|<1$ will satisfy a priori estimate
\begin{equation}\label{alpha0}\|u\|_{L^\infty(\R^N)}+\|v\|_{L^\infty(\R^N)}\leq \frac{\alpha_0}{2}.\end{equation}
  Choose smooth functions
   ${\phi_n}$ such that 
   $\phi_n(0)=\phi_n'(0) =0$, 
   $\phi_n'(t)>0$ for each $t\in(0, 2\alpha_0)$, 
   $\phi_n''(t)>0$ for each $t\in(0, 1/n)$,
   $\phi_n'(t)=1$ for $t\in [1/n,\alpha_0]$,
     $\phi_n'(t)=0$ for $t\geq 2\alpha_0$.
	 We define 
	 \[
	 G_n(s_1, s_2)=G(\phi_n(|s_1|), \phi_n(|s_2|)).
	 \]
    Assumptions (G2) and (G3) ensure that $\partial_{ij}G$ is continuous on $\mathbb{R}^2$ for $i \neq j$. Consequently, the same regularity holds for $\partial_{ij}G_n$, $i\neq j$.
	 By \eqref{eqnabla} and $\phi_n'(0)=0$, it is easy to check that 
	 $\partial_{11}G_n(0, s)=0$ for each $s\in\R$.
	 By the L'Hospital rule, we observe that
	  $  (\phi_n'(t))^2=O(\phi_n (t))$ as $t\to 0^+$.
	 Consequently, using \eqref{sij}, as $(s_1, s_2)\to (0, s)$,  we obtain 
	 \begin{align*}
	 \partial_{11}G_n(s_1, s_2)=&\partial_{11}G(\phi_n(|s_1|), \phi_n(|s_2|))(\phi_n'(|s_1|))^2 +\partial_1 G(\phi_n(|s_1|), \phi_n(|s_2|))\phi_n''(|s_1|)\\
	 =& O(\phi_n^\sigma(|s_1|)) \to 0.
	 \end{align*}
Therefore,   $\partial_{11} G_n$ is continuous at $(0, s)$ for all $s\in\R$.  An analogous argument applies to $\partial_{22}G_n$, yielding continuity at points $(s, 0)$. Then we conclude that $G_n\in C^2(\R^2)$.
Moreover,  it is straightforward to verify that the following conditions hold: 
\begin{enumerate}
    \item[(Gn1)]
	$G_n\in C^2(\R^2)$,  $G_n(s_1, s_2)=G_n(|s_1|, |s_2|)$ for any $(s_1,s_2)\in \R^2$,  $\nabla G_n(s_1, s_2)=(0,0)$ and $G_n(s_1,s_2)=0$ if $s_1=0$ or $s_2=0$;

\item[(Gn2)] For each \( i = 1, 2 \), the partial derivative \( \partial_i G_n(s_1, s_2) \) is increasing in \( s_i \) whenever \( |s_i| \leq \alpha_0  \), \( \partial_i G_n(s_1, s_2) = 0 \) whenever \( |s_i| \geq 2\alpha_0 \);

    \item[(Gn3)] 
 There exists \( s_0 \in (0,\alpha_0) \) such that for all \( n \),    
    \[
    | \partial_{12} G_n(s_1, s_2)| \leq |s_1 s_2|^{\sigma}, 
    \]
	if $(s_1, s_2) \in \big([-\alpha_0, \alpha_0] \times [-s_0, s_0]\big) \cup \big([-s_0, s_0] \times [-\alpha_0, \alpha_0]\big)$,
and	for $i=1,2$, $j\neq i$,
	\[
	| \partial_{ii} G_n(s_1, s_2)|\leq (|s_i|-\frac1n)^{\sigma-1}|s_j|^{\sigma+1}, \mbox{ if } \frac{1}{n}< |s_i|\leq \alpha_0, |s_j|\leq s_0.
	\]

    \item[(Gn4)] For \( i = 1, 2 \), as \( n \to \infty \),
    \[
    G_n(s_1, s_2) \to G(s_1, s_2) \quad \text{and} \quad \partial_i G_n(s_1, s_2) \to \partial_i G(s_1, s_2)
    \]
    uniformly for all \( (s_1, s_2) \in \left[-\alpha_0, \alpha_0\right]^2 \).

    \item[(Gn5)] For     \( i = 1, 2 \) and each \( (s_1, s_2)\in\R^2 \),
    \[
     0\leq  s_i\partial_i G_n(s_1, s_2) \leq s_i\partial_i G(s_1, s_2)  \text{ and }  0 \leq G_n(s_1, s_2) \leq  \min\{G(s_1,s_2), G(2\alpha_0, 2\alpha_0)\}.
    \]
\end{enumerate}
 In what follows, we always assume $n\geq n_\ell$
 for some sufficiently large \begin{equation}\label{nl}
	n_\ell>4(\ln\ell)^5.
 \end{equation}
 
We consider the perturbed   system
\begin{equation}
	\label{eq1.1n}
	\left\{\begin{aligned}-\Delta u+K_1(|x|)u&=\mu |u|^{p-2}u+\beta \partial_1G_n (u, v),\quad x\in\mathbb{R}^N,
	\\
	-\Delta v+K_2(|x|) v&=\nu |v|^{p-2}v+\beta \partial_2 G_n(u,v),\quad x\in\mathbb{R}^N.\end{aligned}\right.
\end{equation}
 
Let 
\[\begin{aligned}
	 I_n\left(u,v\right) =&\frac{1}{2}\int_{\mathbb{R}^N}\big(|\nabla u|^2+K_1(|x|)u^2+|\nabla v|^2+ K_2(|x|)v^2\big)   \\
	 &-\frac1{ p}\int_{\mathbb{R}^N}\bigl(\mu|u|^{p}+\nu|v|^{p}\bigr)-  \beta\int_{\mathbb{R}^N}G_n(u, v),\quad (u,v)\in\mathbb  H_s. 
	\end{aligned}\]
	It is straightforward to verify that $I_n\in C^2(\mathbb H_s)$, and its critical points correspond to solutions of \eqref{eq1.1n}.
Let 	\[ J_n(\varphi,\psi) =I_n\left(U_r+\varphi, V_\rho+\psi\right),\quad (\varphi,\psi)\in \mathbb E. \]
Note that $J_n\in C^2(\mathbb E)$ , and we denote its gradient in   $\mathbb E$ by $\nabla_{\mathbb E} J_n$.

\subsection{Basic estimates}
We consider the following bilinear form
\begin{equation}\label{Hessiann}
	\begin{aligned}
	\left\langle L (u,v),(\varphi,\psi)\right\rangle_\ell &  =\int_{\mathbb{R}^N}\bigl(\nabla u\nabla\varphi+K_1(|x|)u\varphi-(p-1)\mu  U_r^{p-2}u\varphi\bigr)  \\
	&+\int_{\mathbb{R}^N}\left(\nabla v\nabla\psi+ K_1(|x|)v\psi-(p-1)\nu  V_\rho^{p-2}v\psi\right).
	\end{aligned}\end{equation}

By the Riesz representation theorem, $L$ can be regarded as a bounded linear operator on $\mathbb H_s$.	Note that 
$L(u,v)=(L_r u, L_\rho v)$ where each $L_i:H_s\to H_s$, for  $i=r,\rho$, is defined by
\[ 
	\begin{aligned}
	 \langle L_r u, \varphi \rangle  &  =\int_{\mathbb{R}^N}\bigl(\nabla u\nabla\varphi+K_1(|x|)u\varphi-(p-1)\mu  U_r^{p-2}u\varphi\bigr)  \\
	 \langle  L_\rho v, \psi \rangle &=\int_{\mathbb{R}^N}\left(\nabla v\nabla\psi+ K_2(|x|)v\psi-(p-1)\nu  V_\rho^{p-2}v\psi\right).
	\end{aligned}\]

	At the end of this section, we present several key lemmas.
	The first one parallels a result from reference \cite{WY}:
\begin{lem}\label{lemma2.2} There exists $\varrho>0$ independent of $\ell, r$  such that for any $r \in \mathbb{D}_{1}$,  $\rho\in \mathbb D_2$
\[
\|{L}_r u\|  \geq  2{\varrho}\|u\|, u\in E_r,\quad   \|{L}_\rho v\|  \geq  2{\varrho}\|v\|, v\in E_\rho.
\]
\end{lem}
\begin{remark}Note that the additional  symmetry condition in $H_s$ compared to \cite{PW13}, 
 \[
	u(r\cos\theta,r\sin\theta,x_3)=u\left(r\cos\left(\frac{ 2\pi }\ell-\theta\right),r\sin\left( \frac{2\pi }\ell-\theta\right),x''\right)
 \]
 guarantees the validity of $\|{L}_\rho v\|  \geq  2{\varrho}\|v\|$.
\end{remark}

 Next, we recall an   $L^\infty$ estimate from \cite{WZZ} for later utilization:
\begin{lem}\label{lem:2.6}
Assume  $R>0$, 
$k\in \N\setminus\{0\}$, 
$\{u_i\}_{i=1}^k\subset C_0^1(B_R(0))$ with 
$\int_{B_R(0)}u_iu_j=\int_{B_R(0)}u_i=0$ ($i,j=1,\dots,k$, $i\neq j$).
If $\phi\in H^1(B_{R+2}(0))$ satisfies
\[
\int_{B_{R+2}(0)} \nabla \phi\nabla v+b_1\phi v
=\int_{B_{R+2}(0)} b_2 v,\quad\forall v\in H_n,
\]
where
$H_n=\big\{v\in H_0^1(B_{R+2}(0))\ \big|\ \int_{B_R(0)}vu_i=0,\ i=1,\dots,k\big\}
$ and $b_1,b_2\in L^q(B_{R+2}(0))$ for some $q>\frac N2$,
then
\[\norm \phi_{L^\infty(B_{R+1}(0))}\leq C\Big(\norm\phi_{H^1(B_{R+2}(0))}
+\norm {b_2}_{L^q(B_{R+2}(0))}\Big),
\]
where $C>0$ is a constant depending only on $u_i$, $N$, $q$,
$\norm{b_1}_{L^q(B_R(0))}$ and
$\sup_{y\in B_{R+1}(0)}\norm{b_1^-}_{L^q(B_1(y))}$.
Moreover, if $u_i=0$ for all $i=1,\dots,k$, then the constant $C$ depends only on $N$, $q$ and
$\sup_{y\in B_{R+1}(0)}\norm{b_1^-}_{L^q(B_1(y))}$.
\end{lem}
Finally, we state a lemma  concerning the interaction  estimate of the nonlinear term:
\begin{lem}\label{Appendix1}
	Let $r\in \mathbb D_1$, $\rho\in \mathbb D_2$. Then for large $\ell$,
	\begin{align*}&U_r^{p-1}-\sum_{j=1}^\ell  U_{x^j}^{p-1} =O( \ell^{-\frac{m}{2}-2\tau}) \sum_{j=1}^\ell e^{-\frac{p-2}{2} |x-x^j|},\\
		&V_\rho^{p-1}-\sum_{j=1}^{2\ell}  V_{y^j}^{p-1} =O( \ell^{-\frac{m}{2}-2\tau}) \sum_{j=1}^{2\ell} e^{-\frac{p-2}{2} |x-y^j|},
	\end{align*}
	where $\tau \in \left(0,\frac1{16}\min\{1,p-2\}\right)$.
\end{lem}
\begin{proof}
	In $\Omega_1$, we have 
\begin{align*}
	U_r^{p-1}-\sum_{j=1}^\ell  U_{x^j}^{p-1}= O(U_{x^1}^{p-2}\sum_{j=2}^\ell U_{x^j})=O(U_{x^1}^{\frac{p-2}2}\sum_{j=2}^\ell U_{x^j}) e^{-\frac{p-2}{2}|x-x_1|}.
\end{align*}
If $2<p\leq 4$, then for $j\geq 2$,
\begin{align*}
\frac{p-2}{2}|x-x^1|+|x-x^j|\geq& \frac{p-2}{2}|x^1-x^j| + \frac{4-p}{2}|x-x^j|\\
\geq&(p-2)r\sin\frac{\pi(j-1)}{\ell}+\frac{4-p}{2} r\sin\frac{\pi (2j-3)}{\ell}.
\end{align*}
If $ p\geq 4$, then for $j\geq 2$,
\begin{align*}
	\frac{p-2}{2}|x-x^1|+|x-x^j|\geq& |x^1-x^j| = 2r\sin\frac{\pi j}{\ell}.
	\end{align*}
	Hence, 
	\begin{align*}
		&U_{x^1}^\frac{p-2}2  U_{x^j}=O(\ell^{-\min\{\frac{p}{4},1\}m})=O(\ell^{-\frac{m}{2}-2\tau}), j=2,3, \ell-1,\ell\\
		&U_{x^1}^\frac{p-2}2  U_{x^j}=O(\ell^{-4m})=O(\ell^{-\frac{m}{2}-8\tau-1}),4\leq j \leq  \ell-2.
	\end{align*}
	Subsequently, we derive the conclusion for $U_r$, and analogously, for 
	 $V_\rho$.
\end{proof}

\section{ Tail minimization operator}
\label{sec3}

In this section, we further modify the non-linearities to establish and solve a minimization problem associated with the outer regions.

\subsection{Problem setting}
Set 
\[
P_k= \bigcup_{i=1}^\ell B_{k\ln\ln\ell}(x^i),\quad Q_k=\bigcup_{j=1}^{2\ell}B_{k\ln\ln\ell}(y^j),\quad k=1,2,3,4.	
\]
Fix \begin{equation}\label{tau}
    \tau \in \left(0,\min\left\{\frac1{16},\frac{p-2}{16},    \frac{\sigma}{16}\right\}\right).
\end{equation} 
For any $(\varphi,\psi)\in \mathbb E$,
we denote:
\[
\|(\varphi, \psi)\|_{ \infty }= \|\varphi \|_{L^\infty(\R^N)} + \|\psi \|_{L^\infty(\R^N)}.
\]

We introduce  
\[
\Lambda_\ell =\Set{(\varphi,\psi) \in \mathbb E \cap (L^\infty(\R^N))^2   | \ell^{-\frac12}\|(\varphi, \psi)\|+ \|(\varphi, \psi)\|_{ \infty } \leq \ell^{-\frac{m}{2}-\tau}}.
\]

% Our solutions will be find in the set 
% \[
% (U_r, V_\rho) + \Lambda_\ell=\Set{(u_0, v_0) \in \mathbb{E} \cap (L^\infty(\mathbb{R}^N))^2 | (u_0, v_0) - (U_r, V_\rho) \in \Lambda_\ell}.
% \]
Note that   each $(u,v)\in (U_r, V_\rho)+\Lambda_\ell$ satisfies
\eqref{alpha0}.
\medskip
For 
$(\varphi_0, \psi_0)\in  \Lambda_\ell$,
we consider the following   system
   \begin{equation}\label{eqouter}
	\left\{\begin{aligned}&-\Delta u+K_1(|x|)u=\mu |u|^{p-2}u+\beta \partial_1G_n(u, v),\quad &&\mbox{ in }\mathbb{R}^N\setminus P_1,
	\\
	&-\Delta v+ K_2(|x|)v=\nu |v|^{p-2}v+\beta \partial_2G_n(u, v),\quad &&\mbox{ in }\mathbb{R}^N\setminus Q_1,
\end{aligned}\right.
\end{equation}
with 
\begin{equation}\label{S0}
u=\varphi_0+U_r \mbox{ in } P_1,\quad  v= \psi_0 +V_\rho \mbox{ in } Q_1.
\end{equation}
 
%  Making 
% $s_0\in(0, 1)$ in (Gn3) smaller if necessary, we assume  that 
% \[
% 1-(p-1)\mu s_0^{p-2}\geq \frac12, \quad 1-(p-1)\nu s_0^{p-2}\geq \frac12, \quad \sqrt{-\tilde\beta}g_n(2s_0)<\frac12.
% \]
We consider  \eqref{eqouter} satisfying
\begin{equation}\label{s0'}
	\|u\|_{L^\infty(\R^N\setminus P_1)}\leq (\ln\ell)^{-\frac12},\quad  \|v\|_{L^\infty(\R^N\setminus Q_1)} \leq (\ln\ell)^{-\frac12}.
\end{equation}
We will show that for each $(\varphi_0,\psi_0)\in  \Lambda_\ell$, problem \eqref{eqouter} with \eqref{S0} and \eqref{s0'} has a unique solution
$(u,v)$. We define the  set  
 \[
\mathbb S(\varphi_0,\psi_0):= \Set{(u,v)\in \mathbb E |   (u, v) \text{ satisfies } \eqref{S0} \text{ and } \eqref{s0'}},
 \]
 which is closed and convex in $\mathbb E$.

\subsection{A priori decay estimates}
In what follows, we assume $(\varphi_0,\psi_0)\in  \Lambda_\ell$.
If  $(u,v)\in \mathbb S(\varphi_0,\psi_0)$ satisfies \eqref{eqouter}, then
\begin{equation}\label{subsolution}
-\Delta|u|+\frac12|u|\leq 0 \mbox{ in }\R^N\setminus P_1, \quad 	-\Delta|v|+\frac12|v|\leq 0 \mbox{ in }\R^N\setminus Q_1.
\end{equation}
\begin{lem}\label{lemma3.3} 
  If $ (u,v)\in \mathbb E$ satisfies \eqref{S0} and
	\eqref{subsolution}, then
	 there is some $C>0$ independent of   $n, \ell$, and the choice of $(\varphi_0,\psi_0)$ such that
	\begin{equation}\label{decay} 
	|u(x)|\leq C\sum_{j=1}^\ell e^{- \frac1{\sqrt{2}}  |x-x^j|}, \quad 
	|v(x)|\leq C\sum_{j=1}^{2\ell} e^{-  \frac{1}{\sqrt{2}}|x-y^j|}, \quad x\in \R^N.
	 \end{equation}
\end{lem}
\begin{proof}
 
We note that the function
$w_j(x):= e^{-\frac1{\sqrt{2}}|x-x^j|}$
satisfies
\[
-\Delta w_j +	\frac12 w_j \geq 0.
\]
Moreover, for $x\in P_1$, 
\[
|u(x)|=|\varphi_0(x)+U_r(x)|\leq U_r(x)+\ell^{-\frac{m}{2}}\leq 2U_r(x) \leq  C \sum_{j=1}^\ell w_j(x).
\]
Then by comparison principle we  obtain 
\[
|u(x)|\leq C	\sum_{j=1}^\ell w_j(x), \quad x\in \R^N\setminus P_1.
\]
Finally, we get the estimate for $u$. 
The estimate for $v$ follows similarly by analyzing $y^j$-centered exponentials in $\mathbb{R}^N \setminus Q_1$.
\end{proof} 
Next, we modify the estimate the deference between $ (u,v)$ and $(U_r,V_\rho)$.
\begin{lem}\label{lemma3.5}
	Assume $(u,v)\in\mathbb{S}(\varphi_0,\psi_0)$ satisfies \eqref{eqouter}.
	There  is $C>0$   such that
	\begin{align*}
		|u(x)-U_r(x)|&\leq 	C \ell^{-\frac{m}{2}-\tau}(\ln\ell)^{\tau}  \sum_{j=1}^\ell e^{-\tau |x-x^j|},\quad x\in \R^N,\\
		|v(x)-V_\rho(x)| &\leq C\ell^{-\frac{m}{2}-\tau}(\ln\ell)^{\tau} \sum_{j=1}^{2\ell}  e^{-\tau|x-y^j|},\quad x\in \R^N.
	\end{align*}
	
 	Moreover, 
 \[
 \|u(x)-U_r(x)\|_{H^1(\R^N\setminus P_3)} +\|v(x)-V_\rho(x)\|_{H^1(\R^N\setminus Q_3)}\leq  \ell^{-\frac{m-1}{2}-\tau}(\ln\ell)^{-\frac32\tau}.
 \]
\end{lem}

\begin{proof}
	Let $(\varphi,\psi)=(u,v)-(U_r,V_\rho)$.
% 	First by Lemma \ref{lemma3.3}, \eqref{subsolution}, and $\dist(P_2, Q_2)>r/2$, we know 
% \[
% 	\|\varphi\|_{H^1(Q_2)} +\|\psi\|_{H^1(P_2)}\leq\|u\|_{H^1(Q_2)} +\|v\|_{H^1(P_2)}+\|U_r\|_{H^1(Q_2)} +\|V_\rho\|_{H^1(P_2)}\leq  C e^{-\ell}.	
% \]
By  \eqref{eqouter} and \eqref{S0}, we have 
\begin{equation}
	\label{outer''}
	\left\{\begin{aligned}&-\Delta \varphi+K_1\varphi=(1-K_1)U_r +\mu \left(|u|^{p-2}u-\sum_{j=1}^\ell  U_{x^j}^{p-1}\right)+\beta \partial_1G_n(u,v),  &&\mbox{ in }\mathbb{R}^N\setminus P_1,
	\\
	&-\Delta \psi+ K_2\psi=(1-K_2)V_\rho+\nu \left(|v|^{p-2}v-\sum_{j=1}^{2\ell} V_{y^j}^{p-1}\right)+\beta \partial_2G_n(u,v),  &&\mbox{ in }\mathbb{R}^N\setminus Q_1,
\\
&\varphi=\varphi_0 \mbox{ in } P_1,\quad  \psi= \psi_0 \mbox{ in } Q_1.
\end{aligned}\right.
\end{equation}

Noting that $|x-x^i|+ |x-y^j|\geq |x^1-y^1|=O(\ell \ln\ell)$, by (Gn3) and \eqref{decay}, 
we get for some $C>0$,
\begin{equation}\label{eqguv}
	\begin{aligned}
	|\partial_1 G_n(u,v)| \leq \alpha_0 |\partial_{12}G_n(\theta u, v)|
	\leq &
	C  (\sum_{i=1}^\ell\sum_{j=1}^{2\ell} e^{-\frac{1}{\sqrt{2}} (|x-x^i|+ |x-y^j|)})^\sigma\\
	\leq & 
	C  e^{-\ell}\sum_{i=1}^\ell\sum_{j=1}^{2\ell} e^{-\frac{\sigma}{2} (|x-x^i|+ |x-y^j|)},
\end{aligned}
\end{equation}
where   $0<\theta<1$.
By Lemma \ref{Appendix1},
\begin{align*}
	  |u|^{p-2}u-\sum_{j=1}^\ell  U_{x^j}^{p-1}& =|u|^{p-2}u -|U_r|^{p-2}U_r+ U_r^{p-1}-\sum_{j=1}^\ell  U_{x^j}^{p-1}\\
	  &= O(|u|^{p-2}+U_r^{p-2})\varphi +O( \ell^{- \frac{m}{2}-2\tau}) \sum_{j=1}^\ell e^{-\frac{p-2}{2} |x-x^j|}\\
	  &=O((\ln\ell)^{-\frac{p-2}{2}})\varphi+ O( \ell^{- \frac{m}{2}-2\tau}) \sum_{j=1}^\ell e^{-\frac{p-2}{2} |x-x^j|} ,
\end{align*}
and 
\begin{equation}\label{Ku}
|(1-K_1) U_r|= O(\ell^{-m}) \sum_{j=1}^\ell e^{-\frac12|x-x^j|}.	
\end{equation}

Combining the above observation, 
by the first equation in \eqref{outer''}, we find that
$\varphi$ satisfies
\begin{equation}\label{weq}
-\Delta|\varphi|+\frac{1}{2}|\varphi|\leq C \ell^{-\frac{m}{2}-2\tau}\sum_{j=1}^\ell e^{-\tau |x-x^j|}, x\in \R^N\setminus P_1.	
\end{equation}
The comparison principle and  the fact that 
\[
|\varphi(x)|< C \ell^{-\frac{m}{2}-\tau}(\ln\ell)^\tau \sum_{j=1}^\ell e^{-\tau |x-x^j|}, \quad x\in P_1
\]
 imply that 
\[
|\varphi(x)|\leq 	 C \ell^{-\frac{m}{2}-\tau}(\ln\ell)^\tau \sum_{j=1}^\ell e^{-\tau |x-x^j|}, x\in \R^N\setminus P_1.
\] 
The estimate of $\psi$ can be proved by similar arguments.
Finally, let $\eta\in C^1(\R^N)$ be such that 
 $0\leq \eta\leq 1$, $|\nabla \eta|\leq 2$ in $\R^N$, and 
 $\eta=1$ in $\R^N\setminus P_3$, $\eta=0$ in $ \cup_{j=1}^\ell B_{3\ln\ln\ell-1}(x^j)$.
 By 
 \eqref{weq} and the decay estimates, we can obtain 
\begin{equation}\label{eq:weq2}
	 \begin{aligned}
	 	\int_{\R^N\setminus P_3}|\nabla \varphi|^2 + \frac12 \varphi^2\leq &
	\int_{\R^N}|\nabla(\eta \varphi)|^2 + \frac12 (\eta \varphi)^2 
	\\
	 \leq 	
	 &\int_{\R^N}|\nabla \eta|^2 \varphi^2 + C\int  \eta^2|\varphi| \ell^{-\frac{m}{2}-2\tau}\sum_{j=1}^\ell e^{-\tau |x-x^j|}\\
	 \leq & \frac18 \ell^{-m-2\tau+1} (\ln\ell)^{-3\tau}.
	 \end{aligned}
\end{equation}
 This completes the proof.
\end{proof}
Let $(\bar \varphi_0, \bar \psi_0)\in   \Lambda_\ell$ be another point,
and  $(\bar u, \bar v)\in  \mathbb{S}(\bar \varphi_0,\bar \psi_0)$ satisfies \eqref{eqouter}.
	We give   estimate for $(  u,   v)-(\bar u, \bar v)$ expressed in   terms of $(\varphi_0,\psi_0)-(\bar \varphi_0,\bar \psi_0)$.
\begin{lem}\label{lemmawz}
	There is $C >0$
such that  
\begin{align*}
	|w(x)|\leq& C(\ln\ell)^{\tau}\left(\|w_0\|_{L^\infty( P_1)}+  \|z_0\|_{L^\infty(Q_1 )}\right)	\sum_{j=1}^\ell e^{-\tau |x-x^j|},\quad  x\in \R^N\setminus P_1\\
	|z(x)|\leq&  C(\ln\ell)^{\tau}\left(\|w_0\|_{L^\infty( P_1)}+  \|z_0\|_{L^\infty(Q_1 )}\right) \sum_{j=1}^{2\ell}e^{-\tau|x-y^j|}, \quad x\in \R^N\setminus Q_1,
	\end{align*}
	where   $(w,z)=(  u,   v)-(\bar u, \bar v)$ and $(w_0,z_0)=(\varphi_0,\psi_0)-(\bar \varphi_0,\bar \psi_0)$. 
\end{lem}
\begin{proof}
 By  \eqref{eqouter}, we have 
	\begin{equation}
		\label{outerdifference}
		\left\{\begin{aligned}-\Delta w+K_1w=& \mu \left(|u|^{p-2}u-|\bar u|^{p-2}\bar u\right)+\beta (\partial_1G_n(u,v)-\partial_1G_n(\bar u, \bar v))    \mbox{ in }\mathbb{R}^N\setminus P_1,
		\\
		 -\Delta z+ K_2z=&\nu \left(|v|^{p-2}v-|\bar v|^{p-2}\bar v\right)+\beta  (\partial_2G_n(u,v)-\partial_2G_n(\bar u, \bar v))     \mbox{ in }\mathbb{R}^N\setminus Q_1,
	\\
w=&w_0 \mbox{ in } P_1,\quad  z= z_0 \mbox{ in } Q_1.
	\end{aligned}\right.
	\end{equation}
In $\R^N\setminus P_1$, there hold
\[
	|u|^{p-2}u-|\bar u|^{p-2}\bar u=O(|u|^{p-2}+|\bar u|^{p-2})w=O((\ln\ell)^{-\frac{p-2}{2}})w,
\]
and by (Gn3)
\begin{align*}
	&\partial_1G_n(u,v)-\partial_1G_n(\bar u, \bar v)=\partial_1G_n(u,v)-\partial_1G_n(\bar u,   v)+\partial_1G_n(\bar u,  v)-\partial_1G_n(\bar u, \bar v),\\
		=&\partial_1G_n(u,v)-\partial_1G_n(\bar u,   v)+O(e^{-\ell}\sum_{i=1}^\ell\sum_{j=1}^{2\ell} e^{-\frac{\sigma}2 (|x-x^i|+ |x-y^j|)}) z.
\end{align*}
Then by Kato's inequality, and noting that $ \beta (\partial_1G_n(u,v)-\partial_1G_n(\bar u,   v)) w\leq 0$,
we have 
\[
	-\Delta|w|+\frac12 |w|\leq Ce^{-\ell}\sum_{j=1}^\ell e^{-\frac{\sigma}{2}   |x-x^j|} |z|, \mbox{ in } \R^N\setminus P_1.
\]
Similarly,
\[ 
	-\Delta|z|+\frac12 |z|\leq Ce^{-\ell}\sum_{j=1}^{2\ell}e^{-\frac{\sigma}{2}|x-y^j|} |w|, \mbox{ in } \R^N\setminus Q_1.
\]
Through the application of the comparison principle, we derive
\begin{equation}\label{eq3.8}
	\begin{aligned}
|w(x)|\leq C((\ln\ell)^{\tau}\|w_0\|_{L^\infty(  P_1)}+ e^{-\ell}  \|z\|_{L^\infty(\R^N )})	\sum_{j=1}^\ell e^{-\tau |x-x^j|}, x\in \R^N\setminus P_1,\\
|z(x)|\leq  C((\ln\ell)^{\tau}\|z_0\|_{L^\infty(  Q_1)}+ e^{-\ell}  \|w\|_{L^\infty(\R^N )})\sum_{j=1}^{2\ell} e^{-\tau|x-y^j|}, x\in \R^N\setminus Q_1,
\end{aligned}
\end{equation}
which lead to
\begin{align*}
\|w\|_{L^\infty(\R^N\setminus P_1)}\leq& C(\|w_0\|_{L^\infty(  P_1)} + e^{-\ell}  \|z\|_{L^\infty(\R^N )})\\
=&C(\|w_0\|_{L^\infty(  P_1)} + e^{-\ell} \|z_0\|_{L^\infty(Q_1)}+ e^{-\ell} \|z\|_{L^\infty(\R^N\setminus Q_1)}),\\
\|z\|_{L^\infty(\R^N\setminus Q_1)}\leq &C(\|z_0\|_{L^\infty( Q_1)} + e^{-\ell} \|w_0\|_{L^\infty(P_1)} + e^{-\ell} \|w\|_{L^\infty(\R^N\setminus P_1)}).	
\end{align*}
Upon combining the estimates, we obtain
\[
	\| w\|_{L^\infty(\R^N\setminus P_1)} + \| z\|_{L^\infty(\R^N\setminus Q_1)}\leq C  (  
	 \| w_0\|_{L^\infty( P_1)} + \| z_0\|_{L^\infty( Q_1)}).	
\]
By inserting this result into   \eqref{eq3.8}, we conclude our desired assertion on $|w(x)|$ and $ |z(x)|$.
\end{proof}

\subsection{Existence of solution to \eqref{eqouter}}\label{sec3.1}
It follows from Lemma \ref{lemmawz} that, for each $(\varphi_0, \psi_0) \in \Lambda_\ell$, the problem \eqref{eqouter} admits at most one solution in $\mathbb{S}(\varphi_0, \psi_0)$.  
We now turn to the proof of existence.  
To this end, we consider the minimization problem  
\begin{equation}\label{min}
\inf \left\{ I_n(u, v) \mid (u, v) \in \mathbb{S}(\varphi_0, \psi_0) \right\}.
\end{equation}
 \begin{lem}\label{lemma3.2}
	The problem \eqref{min} is attained by a unique solution $(u,v) $ of \eqref{eqouter} in $\mathbb{S}(\varphi_0, \psi_0)$.
\end{lem}
\begin{proof}
First note that $\mathbb{S}(\varphi_0,\psi_0)$ is weakly closed. To show that the problem 
\eqref{min} 
is attained, it suffices to prove that $I_n(u,v)$ is weakly lower semicontinuous and coercive on $\mathbb{S}(\varphi_0,\psi_0)$. 

To verify weak lower semicontinuity, observe that 
$$
I_n(u,v) = I_n(u,v) + \beta \int_{\mathbb{R}^N} G_n(u,v) - \beta \int_{\mathbb{R}^N} G_n(u,v).
$$
The term $I_n(u,v) + \beta \int_{\mathbb{R}^N} G_n(u,v)$ is convex on $\mathbb S(\varphi_0,\psi_0)$ and hence weakly lower semicontinuous. On the other hand, since $G_n \geq 0$, Fatou's Lemma implies that $\int_{\mathbb{R}^N} G_n(u,v)$ is weakly lower semicontinuous. 

To establish coercivity, we observe that for  
$
C_0 = \frac{1}{p} \int_{P_1} \mu |u_0|^p + \frac{1}{p} \int_{Q_1} \nu |v_0|^p,
$
it follows from \eqref{s0'} that  
\begin{align*}
I_n(u, v) 
&\geq \frac{1}{2} \|(u, v)\|^2 - \frac{1}{p} \int_{\mathbb{R}^N \setminus P_1} \mu |u|^p - \frac{1}{p} \int_{\mathbb{R}^N \setminus Q_1} \nu |v|^p - C_0 \\
&\geq \frac{1}{2} \|(u, v)\|^2 - C (\ln \ell)^{-\frac{p-2}{2}} \|(u, v)\|^2 - C_0.
\end{align*}
As $\|(u, v)\| \to \infty$, the right-hand side tends to $+\infty$, which confirms the coercivity of $I_n$.  
Therefore, the minimization problem \eqref{min} is attained by some $(u, v) \in \mathbb{S}(\varphi_0,\psi_0)$.

Next, we show that this minimizer $(u, v) \in \mathbb{S}(\varphi_0,\psi_0)$ is a solution to \eqref{eqouter}.  
Since $I_n(|u|, |v|) = I_n(u, v)$ and $(|u|, |v|) \in \mathbb{S}(\varphi_0,\psi_0)$, it follows that $(|u|, |v|)$ is also a minimizer.  
Now, for any non-negative test functions $\xi \in C_0^\infty(\mathbb{R}^N \setminus P_1)$ and $\eta \in C_0^\infty(\mathbb{R}^N \setminus Q_1)$ with $\xi, \eta \ge 0$, we have  
$$
(|u| - t\xi, |v| - t\eta) \in \mathbb{S}(\varphi_0,\psi_0)
$$  
for sufficiently small $t > 0$.  

Thus,
$$
\langle \nabla I_n(|u|, |v|), (\xi, \eta) \rangle 
= \lim_{t \to 0^+} \frac{I_n(|u|, |v|) - I_n(|u| - t\xi, |v| - t\eta)}{t} \leq 0.
$$

This implies that $(|u|, |v|)$ satisfies
$$
\left\{
\begin{aligned}
&-\Delta |u| + K_1(|x|)|u| \leq \mu |u|^{p-1} + \beta \partial_1 G_n(|u|, |v|), && \text{in } \mathbb{R}^N \setminus P_1, \\
&-\Delta |v| + K_2(|x|)|v| \leq  \nu |v|^{p-1} + \beta \partial_2 G_n(|u|, |v|), && \text{in } \mathbb{R}^N \setminus Q_1.
\end{aligned}
\right.
$$

By \eqref{s0'} and $\beta<0$, we conclude that $(u,v)\in \mathbb S(\varphi_0,\psi_0)$ satisfies \eqref{subsolution}.
Then by Lemma \ref{lemma3.3}, we know that 
the inequalities in \eqref{s0'} holds strictly for $(u,v)$. 
Therefore, $(u,v)\in \mathbb S(\varphi_0,\psi_0)$ is a solution to \eqref{eqouter}.
\end{proof}
We now define the tail minimization operator as follows:
\begin{defn}  Given any \( (\varphi_0,\psi_0)\in \Lambda_\ell \), let 
  $(u,v)\in \mathbb S(\varphi_0,\psi_0)$ denote the unique solution to \eqref{eqouter}.
The  tail minimization operator  $ S: \Lambda_\ell \to \mathbb{E} $ is then defined by  
	\[
	S(\varphi_0,\psi_0):=(u,v)-(U_r, V_\rho).
	\]
\end{defn}

\section{Reduction}

In this section, we will develop a reduced framework to construct the concentrated solution of the system of equations \eqref{eq0'}. As explained in the introduction, for problems with sublinear coupling nonlinear terms, the difficulty of directly applying perturbation methods for reduction is quite evident. Proving the existence of a fixed point by contraction mapping in $\mathbb{E}$ is problematic, and ultimately, using the implicit function theorem to prove the smoothness of the concentrated solution with respect to the parameters is also infeasible. To overcome these difficulties, we take advantage of the solvability in the outer region, where the sublinear terms present challenges, and instead impose the fixed point problem within the set $S\Lambda_\ell$.

Our purpose in this section is to find  $(\varphi, \psi)$ such that 
$\nabla_{\mathbb E}  J_n(\varphi,\psi)=0$ in $\mathbb E$.
Equivalently, we solve
\[L(\varphi,\psi)=
\gamma+ R(\varphi, \psi) +N_n(\varphi,\psi) \text{ in }\mathbb E,
\]
where $\gamma$, $R(\varphi, \psi)$ and $N_n(\varphi,\psi)$ are respectively  defined by
\begin{align*}
	\langle \gamma,(w,z)	\rangle=&\int_{\R^N} (\gamma_1,\gamma_2)\cdot(w,z),\\
\langle R(\varphi, \psi), (w,z)\rangle=	
&\int_{\R^N}  (R_{1}(\varphi),  R_{2}(\psi))\cdot(w,z),\\ 
\langle N_n(\varphi, \psi), (w,z)\rangle=&\beta\int_{\R^N} \nabla G_n(U_r+\varphi,V_\rho+\psi) \cdot(w,z)
\end{align*}
for any $(w,z)\in\mathbb E$, where
\begin{gather*}
	\gamma_1=\mu(U_r^{p-1}-\sum_{j=1}^\ell U_{x^j}^{p-1}+(1-K_1)U_r), \quad \gamma_2=\nu(V_\rho^{p-1} - \sum_{j=1}^{2\ell}V_{y^j}^{p-1}+(1-K_2)V_\rho),\\
	R_{1}(\varphi)= \mu(|U_r+\varphi|^{p-2}(U_r+\varphi) -U_r^{p-1}-(p-1) U_r^{p-2}\varphi),\\
	R_{2}(\psi)=\nu(|V_\rho+\psi|^{p-2}(V_\rho+\psi) -V_\rho^{p-1}-(p-1) V_\rho^{p-2}\psi).
\end{gather*}

We emphasize that directly addressing the fixed point problem on $\Lambda_\ell$ via the operator  
$$
(\varphi, \psi) = T(\varphi, \psi) := L^{-1}\left( \gamma + R(\varphi, \psi) + N_n(\varphi, \psi) \right)
$$  
is generally not feasible, since $T$ fails to be a contraction on $\Lambda_\ell$.  

However, observe that if $(\varphi, \psi) \in \Lambda_\ell$ is a fixed point of $T$, then $(u, v) = (\varphi, \psi) + (U_r, V_\rho)$ solves \eqref{eqouter}. By uniqueness, it follows that $(\varphi, \psi) = S(\varphi, \psi) \in S\Lambda_\ell$.  
This observation enables us to focus our attention on the fixed point problem within the set  $S\Lambda_\ell$.
 
 Since every element $ (\varphi, \psi) \in S\Lambda_\ell $ is uniquely determined by its restrictions $ \varphi|_{P_1} $ and $ \psi|_{Q_1} $, we define a metric on $S\Lambda_\ell$ by
\[
d((\varphi,\psi),(\varphi',\psi'))=\|\varphi-\varphi'\|_{P}+\|\psi-\psi'\|_{Q},
\]
where for $\xi\in H^1(P_1)\cap L^\infty(P_1)$, $\eta\in H^1(Q_1)\cap L^\infty(Q_1)$,
\[ \|\xi\|_{P}=
\ell^{-\frac12}\|\xi\|_{H^1(P_1)} +\|\xi\|_{L^\infty(P_1)},
\]
\[ \|\eta\|_{Q}=
\ell^{-\frac12}\| \eta\|_{H^1(Q_1)} +\| \eta\|_{L^\infty(Q_1)}.
\]
Equipped with this metric, $ (S\Lambda_\ell, d) $ becomes a complete metric space.  
This construction allows us to isometrically identify $ S\Lambda_\ell $ with the space  
\[\{
(\varphi|_{P_1}, \psi|_{Q_1})\ |\ (\varphi,\psi)\in \Lambda_\ell\}.
\]
Crucially, even though $S\Lambda_\ell$ may not be a subset of $\Lambda_\ell$, the composition $S \circ S$ is well-defined and satisfies $S \circ S = S$. In particular, if $(\varphi, \psi) \in S\Lambda_\ell$, then $S(\varphi, \psi) = (\varphi, \psi)$.
For each $(\varphi,\psi)\in S\Lambda_\ell$, 
 and $(\xi,\eta)\in H_0^1(\R^N\setminus P_1)\times H_0^1(\R^N\setminus Q_1) \cap \mathbb H_s$,
there hold 
\begin{equation}\label{eq 4.1}
	\begin{aligned}
	   \int_{\R^N}\left(\gamma_1+R_{1}(\varphi)+\beta\partial_1G_n(\varphi+U_r, \psi+V_\rho)  \right)\xi &=\langle L_r (\varphi), \xi \rangle,\\
	   \int_{\R^N}\left( \gamma_2+R_{2}(\psi)+\beta\partial_2G_n(\varphi+U_r, \psi+V_\rho) \right)\eta &=\langle L_\rho (\psi), \eta \rangle.
	\end{aligned}
\end{equation}
 For any $\zeta, \xi\in E_r$, $C^1$ truncation $\chi$, we have 
\begin{align}\label{zeta}
\langle L_r (\chi \zeta), \xi \rangle  
=\langle L_r \zeta, \chi\xi \rangle +\int_{\R^N} \zeta(2\nabla \chi\nabla\xi+  \Delta \chi \xi )
\end{align}
In general, $T$ does not map $S\Lambda_\ell$ to $S\Lambda_\ell$. However, we can show that
\begin{lem}\label{lemma4.3}  
	$T$ maps $S\Lambda_\ell$ to  $\Lambda_\ell$, and for each $(\varphi, \psi)\in S\Lambda_\ell$, $(\bar \varphi, \bar \psi)\in S\Lambda_\ell$,
	\begin{gather}\label{T11}
		 \ell^{-\frac12} \|T(\varphi, \psi)\|+\|T(\varphi, \psi)\|_\infty=o((\ln\ell)^{-\tau}\ell^{-\frac{m}{2}-\tau}),\\
	  d(ST(\varphi, \psi), ST (\bar \varphi, \bar \psi))= 
	 o((\ln\ell)^{-\tau})(\|\varphi-\bar \varphi\|_{L^\infty( P_1)}+  \|\psi-\bar \psi\|_{L^\infty(Q_1 )}).\label{T12}
	\end{gather}
\end{lem}
\begin{proof}
(i) Let $(\varphi, \psi) \in S\Lambda_\ell$, and denote $(\varphi_0, \psi_0) = T(\varphi, \psi)$, $(u, v) = (\varphi, \psi) + (U_r, V_\rho)$. We aim to verify \eqref{T11}.
For any $(\xi, 0) \in \mathbb{E}$, we have  
\begin{equation}\label{eq:int1}
\int_{\R^N}(\gamma_1   +R_1(\varphi) +\beta\partial_1 G_n(u,v))\xi=\langle L_r\varphi_0, \xi\rangle.
\end{equation}
If we assume further that $\xi\in H_0^1(\R^N\setminus P_1)$,
then by \eqref{eq 4.1}, it follows that  
\begin{equation}\label{eq 4.2}
	\langle L_r(\varphi_0-\varphi ), \xi\rangle=0.
\end{equation}
Since 
for large $\ell$,  $(p-1)U_r^{p-2} < \frac{1}{2}$ in $\mathbb{R}^N \setminus P_1$,  it follows that   $\varphi_0-\varphi $ satisfies the inequalities:
\begin{equation}\label{eq 4.5}
-\Delta|\varphi_0-\varphi| +\frac12|\varphi_0-\varphi|\leq 0\mbox{ in } \R^N\setminus P_1.
\end{equation}
Consequently, we obtain
\[
|\varphi_0-\varphi|(x)\leq C(\ln\ell)^\tau \|\varphi_0-\varphi\|_{L^\infty(P_1)}\sum_{j=1}^\ell e^{-\tau|x-x^j|}, x \in \R^N\setminus P_1.\]
By Lemma \ref{lemma3.5} and $\|\varphi\|_{P} \leq  \ell^{-\frac{m}{2}-\tau}$,
we know 
\begin{equation}\label{eq 4.3}
	|\varphi_0(x)|\leq C(\ln\ell)^\tau (\|\varphi_0\|_{L^\infty(P_1)}+\ell^{-\frac{m}{2}-\tau})
	\sum_{j=1}^\ell e^{-\tau|x-x^j|}, x \in \R^N\setminus P_1.
\end{equation}

Now, let $\chi \in C^1(\mathbb{R}^N)$ be a cutoff function satisfying  
\begin{equation}\label{chi}
	\begin{cases}
	\chi = 0   \text{ in } P_3,\quad \chi = 1   \text{ in } \mathbb{R}^N \setminus P_4, \\
	0 \leq \chi \leq 1,\quad  |\nabla \chi| \leq 2/\ln\ln\ell,\quad  |\Delta \chi|\leq   8N/(\ln\ln\ell)^2  \text{ in } \mathbb{R}^N.
	\end{cases}
\end{equation}
By \eqref{zeta}, \eqref{eq 4.2}, and Lemma \ref{lemma3.5}, for any $\xi \in E_r$, we compute  
\begin{align*}
	\langle L_r (\varphi_0-\chi \varphi), \xi \rangle
	=&\langle L_r \varphi_0, \xi \rangle-	\langle L_r \varphi, \chi\xi \rangle 
	+\int_{\R^N}\varphi (2\nabla \chi\nabla\xi+ \Delta \chi \xi)\\
	=&\langle L_r \varphi_0, (1-\chi)\xi \rangle+	\langle L_r(\varphi_0- \varphi), \chi\xi \rangle +o( \|\varphi\|_{L^2(P_4\setminus P_3)})\|\xi\|\\
	=&\langle L_r \varphi_0, (1-\chi)\xi \rangle +o((\ln\ell)^{-\tau} \ell^{-\frac{m-1}{2}-\tau})\|\xi\|.
	\end{align*}
Since $\supp (1-\chi) \subset P_4$, we conclude   
\begin{align}\label{1'}
	\begin{split}
	& \|L_r (\varphi_0-\chi \varphi)\|\\
	\leq&\|\gamma_1\|_{L^2(P_4)}+ \| R_{1}(\varphi)\|_{L^2(P_4)} +|\beta| \|\partial_1G_n(u,v) \|_{L^2(P_4)}+o((\ln\ell)^{-\tau} \ell^{-\frac{m-1}{2}-\tau})\\
	\leq & O(\ell^\frac12 (\ln\ln\ell)^\frac{N}{2}) (\|\gamma_1\|_{L^\infty(P_4)}+ \| R_{1}(\varphi)\|_{L^\infty(P_4)} +  \|\partial_1G_n(u,v) \|_{L^\infty(P_4)})\\
	&+o((\ln\ell)^{-\tau} \ell^{-\frac{m-1}{2}-\tau}).
	\end{split}
\end{align}
From Lemma \ref{Appendix1} and \eqref{Ku}, we know $\|\gamma_1\|_{L^\infty(P_4)} = O(\ell^{-\frac{m}{2}-2\tau})$. Since $|U_r| \geq (\ln \ell)^{-5}$ and $|\varphi| \leq (\ln \ell)^{\tau} \ell^{-\frac{m}{2}-\tau}$ in $P_4$, we estimate  for some $\theta\in(0,1)$,
\[
|R_{1}(\varphi)|\leq C|(U+\theta \varphi)^{p-2}-U^{p-2}||\varphi| \leq C(\ln\ell)^{5(3-p)^+}\varphi^2=o(\ell^{-m})\quad 
\text{ in } P_4.
\]
Moreover, by (Gn3), we have in $P_4$,
\[|\partial_1G_n(u,v)  |
=O( |\partial_{12}G_n(\theta u, v)|)=O(|uv|^\sigma)=O(e^{-\ell}).
\]
Therefore,  
\[
\|\varphi_0\|_{H^1(P_3)}\leq 
\|\varphi_0-\chi \varphi\|\leq 
\|L_r (\varphi_0-\chi \varphi)\|=o((\ln\ell)^{-\tau} \ell^{-\frac{m-1}{2}-\tau})
\]
Using \eqref{eq:int1} and Lemma \ref{lem:2.6}, for each $y\in \cup_{j=1}^\ell B_{3\ln\ln\ell-2}(x^j)$, we  obtain
\begin{align*}
&\|\varphi_0\|_{L^\infty (B_{1}(y))}\\
\leq& C(  \| \varphi_0\|_{H^1(B_{2}(y))} +\|\gamma_1\|_{L^\infty(B_{2}(y))}+ \|R_{1}(\varphi) \|_{L^\infty(B_{2}(y))}  + \|\partial_1G_n(u,v) \|_{L^\infty(B_{2}(y))} )\\
\leq & C \ell^{-\frac12} \| \varphi_0\|_{H^1(P_3)} +  o(  \ell^{-\frac{m}{2}-2\tau})=o((\ln\ell)^{-\tau} \ell^{-\frac{m}{2}-\tau}).
\end{align*}
Thus,   \[\|\varphi_0\|_{L^\infty (\cup_{j=1}^\ell B_{3\ln\ln\ell-2}(x^j))}=o((\ln\ell)^{-\tau} \ell^{-\frac{m}{2}-\tau}).\]
Recalling 
\eqref{eq 4.3}, we also have  
\[
\|\varphi_0\|_{L^\infty (\R^N\setminus\cup_{j=1}^\ell B_{3\ln\ln\ell-2}(x^j))}=o((\ln\ell)^{-\tau} \ell^{-\frac{m}{2}-\tau}).
\]
Therefore,
\[
\|\varphi_0\|_{L^\infty(\R^N)}=o((\ln\ell)^{-\tau} \ell^{-\frac{m}{2}-\tau}).
\]
Using \eqref{eq:4.5} and similar arguments as in \eqref{eq:weq2}, we derive  
\[
 \int_{\R^N\setminus P_3}|\nabla \varphi_0-\nabla\varphi|^2 + \frac12 |\varphi_0-\varphi|^2\leq 
 \int |\nabla \eta|^2 |\varphi_0-\varphi|^2= O (\ell^{-m-2\tau+1} (\ln\ell)^{-3\tau}),
\]
where $\eta$ is given as in \eqref{eq:weq2}.
Combining this with Lemma \ref{lemma3.5}, we obtain  
\[
\|\varphi_0\|_{H^1(\R^N)} \leq  \|\varphi_0\|_{H^1(P_3)}+ \|\varphi_0-\varphi\|_{H^1(\R^N\setminus P_3)}
+\|\varphi\|_{H^1(\R^N\setminus P_3)}=o((\ln\ell)^{-\tau} \ell^{-\frac{m-1}{2}-\tau}).
\]
An analogous result holds for $\psi_0$. Therefore, we obtain   \eqref{T11} and $T:S\Lambda_\ell\to \Lambda_\ell$.

(ii) Let $(\bar \varphi, \bar \psi)\in S\Lambda_\ell$ be another point,  We then show \eqref{T12} 
 Set
	\begin{gather*}
		(u, v) = (\varphi, \psi) + (U_r, V_\rho), (\bar u,\bar v) =(\bar \varphi, \bar \psi)+ (U_r,V_\rho),\\
		(w,z)=	(\varphi,\psi)- (\bar\varphi,\bar\psi)=(u,v)-(\bar u,\bar v),\\
	(W,Z)=T (\varphi, \psi)-T (\bar \varphi, \bar \psi).
	\end{gather*}
	Then  by the definition of $S$, $d(ST(\varphi, \psi), ST (\bar \varphi, \bar \psi))=\|W\|_P+\|Z\|_Q$.
We have for   $(\xi,0)\in\mathbb E$,
\begin{equation} \label{eq4.3}
	\begin{aligned}
	   \int_{\R^N}\left( R_{1}(\varphi)-R_1(\bar \varphi)+\beta\partial_1G_n(u,v)-\beta\partial_1G_n(\bar u,\bar v)\right)\xi &=\langle L_r W, \xi \rangle.
	   	%    \int_{\R^N}\left(  R_{2}(v-V_\rho)-R_2(\bar v-V_\rho)+G_n(u) g_n(v)-G_n(\bar u) g_n(\bar v) \right)\eta &=\langle L_\rho Z, \eta \rangle.
	\end{aligned}
\end{equation}
From \eqref{eq 4.2}
for   $(\xi, 0)\in \mathbb E$ with $\xi\in H_0^1(\R^N\setminus P_1)$,
 we obtain
\[
\langle L_r (W-w), \xi \rangle =0
 \text{
for any  }\xi\in H_0^1(\R^N\setminus P_1)\cap H_s.\]
% Considering that, for large $\ell$,  $(p-1)U_r^{p-2} < \frac{1}{2}$ in $\mathbb{R}^N \setminus P_1$, it follows that the pair $W-w$ satisfies  
% \[
% -\Delta|W-w| +\frac12|W-w|\leq 0\mbox{ in } \R^N\setminus P_1.
% \]
% By the comparison principle, we have 
% \begin{equation}\label{eq4.5}
% 	\begin{aligned}
% 	|W(x)-w(x)|\leq &C(\ln \ell)^\frac12\|W-w\|_{L^\infty(P_1)} \sum_{j=1}^\ell e^{-\frac{1}{2}|x-x^j|}, &&x\in \R^N\setminus P_1.
% 	% |Z(x)-z(x)|\leq &C(\ln \ell)^\frac12\|Z-z\|_{L^\infty(Q_1)} \sum_{j=1}^{2\ell} e^{-\frac{1}{2}|x-y^j|}, &&x\in \R^N\setminus Q_1.
% 	\end{aligned}
% \end{equation}
 
Then by \eqref{zeta}, \eqref{eq4.3}, and Lemma \ref{lemmawz}, we have for $\xi\in E_r$,
\begin{align*}
	\langle L_r (W-\chi w), \xi \rangle=&\langle L_r W, \xi \rangle-	\langle L_r w, \chi\xi \rangle +\int_{\R^N} w(2\nabla \chi\nabla\xi+\Delta \chi \xi)\\
	=&\langle L_r W, (1-\chi)\xi \rangle+	\langle L_r(W- w), \chi\xi \rangle +o( \|w\|_{L^2(P4\setminus P_3)})\|\xi\|\\
	=&\langle L_r W, (1-\chi)\xi \rangle +o((\ln\ell)^{-\tau}) \ell^\frac12	\left( \|w\|_{L^\infty( P_1)}+  \|z\|_{L^\infty(Q_1 )}\right)\|\xi\|,
	\end{align*}
	where $\chi$ is the cutoff function satisfying \eqref{chi}.
Since $\supp (1-\chi) \subset P_4$, we obtain 
\begin{align}\label{1}
	\begin{split}
	& \|L_r (W-\chi w)\|\\
	\leq& \| R_{1}(\varphi)-R_1(\bar \varphi)\|_{L^2(P_4)} +|\beta| \|\partial_1G_n(u,v)-\partial_1G_n(\bar u,\bar v)\|_{L^2(P_4)}\\
	&\quad+o((\ln\ell)^{-\tau}) \ell^\frac12	\left( \|w\|_{L^\infty( P_1)}+  \|z\|_{L^\infty(Q_1 )}\right)\\
	\leq &C\ell^\frac12(\ln\ln\ell)^\frac{N}{2}\big(\| R_{1}(\varphi)-R_1(\bar \varphi)\|_{L^\infty(P_4)} +
	  \|\partial_1G_n(u,v)-\partial_1G_n(\bar u,\bar v)\|_{L^\infty(P_4)}\big) 
	  \\& \quad+o((\ln\ell)^{-\tau}) \ell^\frac12	\left( \|w\|_{L^\infty( P_1)}+  \|z\|_{L^\infty(Q_1 )}\right).\end{split}
\end{align}

Note that  $\inf_{P_4}  U_r>(\ln\ell)^{-5}$. For some $\tilde\varphi$ satisfying $|\tilde\varphi|\leq |\varphi|+|\bar \varphi|$, by
Lemma \ref{lemma3.5}, we have
	\begin{align*}
	\| R_{1}(\varphi)-R_1(\bar \varphi)\|_{L^\infty(P_4)} 
	\leq & C\|(|U+\tilde\varphi|^{p-2}-U^{p-2} )w\|_{L^\infty(P_4)}\\
	\leq& C(\ln\ell)^{ 5(3-p)^+} \| \tilde\varphi w\|_{L^\infty(P_4)}\\
	=& o(\ell^{-\frac{m}{2}})\|   w\|_{L^\infty(P_4)}.
	\end{align*}
	where $(3-p)^+=\max\set{3-p,0}$.
 To estimate $\|\partial_1G_n(u,v)-\partial_1G_n(\bar u,\bar v)\|_{L^\infty(P_4)}$, we note that 
\begin{align*}
	 |\partial_1G_n(u,v)-\partial_1G_n(\bar u,\bar v)|	 
	\leq  |\partial_1G_n(u,v)-\partial_1G_n( u,\bar v)| +|\partial_1G_n(u,\bar v)-\partial_1G_n(\bar u,\bar v)|.
\end{align*}
Since $\dist(y^j, P_4)\geq C\ell\ln\ell$,
 we have by Lemma \ref{lemma3.3}
  $\|v\|_{L^\infty(P_4)}=o(e^{- \ell})$ and $\inf_{P_4}u\geq (\ln \ell)^{-5}$.
 Using (Gn3) and \eqref{nl}, we estimate
\begin{align*}
	\|\partial_1G_n(u,v)-\partial_1G_n( u,\bar v)\|_{L^\infty(P_4)} \leq &	C\sup_{P_4} (|uv|^\sigma +|u\bar v|^\sigma) |z| 
	\leq     e^{- \ell}\|   z\|_{L^\infty(P_4)},\\
	\|\partial_1G_n(u,\bar v)-\partial_1G_n(\bar u,\bar v)\|_{L^\infty(P_4)}\leq &C\sup_{P_4}|v|^{\sigma+1} |(\ln\ell)^{-5}-\frac1n|^{\sigma-1} |w| 
	\leq   e^{-  \ell} \|w\|_{L^\infty(P_4)}.
\end{align*}
By Lemma \ref{lemmawz} we have 
\[
\|(w,z)\|_\infty =O (\|w\|_{L^\infty(P_1)}+\|z\|_{L^\infty(Q_1)}).
\]
Therefore,
in view of \eqref{1}  we conclude that
\begin{align*}
	\|L_r (W-\chi w)\|= o((\ln\ell)^{-\tau}) \ell^\frac12	\left( \|w\|_{L^\infty( P_1)}+  \|z\|_{L^\infty(Q_1 )}\right).
\end{align*}
Furthermore, by Lemma \ref{lemma2.2},
\begin{equation}\label{eq:4.5}
	\begin{aligned}
		\|W\|_{H^1(P_3)}&\leq \|W-\chi w\| 
		\leq C\|L_r (W-\chi w)\| \\
		& = o((\ln\ell)^{-\tau}) \ell^\frac12	\left( \|w\|_{L^\infty( P_1)}+  \|z\|_{L^\infty(Q_1 )}\right).
	\end{aligned}
\end{equation}
Using \eqref{eq4.3} and Lemma \ref{lem:2.6}, for each $y\in P_2$, we  obtain
\begin{align*}
\|W\|_{L^\infty (B_{1}(y))}\leq& C(  \| W\|_{H^1(B_{2}(y))} + \|R_{1}(\varphi)-R_1(\bar \varphi)\|_{L^\infty(B_{2}(y))} \\
&+ \|\partial_1G_n(u,v)-\partial_1G_n(\bar u,\bar v)\|_{L^\infty(B_{2}(y))} )\\
=& O(\ell^{-\frac12})\| W\|_{H^1(P_3)}+  o((\ln\ell)^{-\tau}) 	\left( \|w\|_{L^\infty( P_1)}+  \|z\|_{L^\infty(Q_1 )}\right)\\
=&o((\ln\ell)^{-\tau}) 	\left( \|w\|_{L^\infty( P_1)}+  \|z\|_{L^\infty(Q_1 )}\right),
\end{align*}
which implies
\[
\|W\|_{L^\infty (P_2)}	=   o((\ln\ell)^{-\tau}) 	\left( \|w\|_{L^\infty( P_1)}+  \|z\|_{L^\infty(Q_1 )}\right).
\]
Therefore,
\[
\|W\|_{P}	=   o((\ln\ell)^{-\tau}) 	\left( \|w\|_{L^\infty( P_1)}+  \|z\|_{L^\infty(Q_1 )}\right).
\]
Similarly,
we have the estimate for $Z$ and  the proof is complete.
\end{proof}
\medskip
Lemma \ref{lemma4.3}   implies $ST:S\Lambda_\ell\to S\Lambda_\ell$ is well-defined and 
is a contraction map. Therefore, it has a unique fixed point in $S\Lambda_\ell$.
\begin{lem}\label{lemma4.1}
	$(\varphi, \psi)$ is a fixed point of   $ST$ in $S\Lambda_\ell$ if and only if it is a fixed point of $T$ in $\Lambda_\ell$. Furthermore, any such fixed point satisfies $(\varphi, \psi)\in S\Lambda_\ell \cap \Lambda_\ell$
	and $(\varphi, \psi)=S(\varphi, \psi)=T(\varphi, \psi)$.
\end{lem}
\begin{proof}
	We have observed that 
	if $(\varphi,\psi)\in\Lambda_\ell$ is a fixed point of $T$, then 
	$(\varphi,\psi)\in S\Lambda_\ell$. 
	Hence 
	\[
	(\varphi,\psi)=S(\varphi,\psi)=ST(\varphi,\psi).
	\]
	Therefore, it is a fixed point of $ST$ in $S\Lambda_\ell$.

	Now assume  that $(\varphi,\psi)\in S\Lambda_\ell$
	is a fixed point of $ST$. By \eqref{T11} and  Lemma \ref{lemma3.5},
	we know $(\varphi,\psi)=ST(\varphi,\psi)\in \Lambda_\ell$
	for large $\ell$. Furthermore,
	denoting 
$(\varphi_0, \psi_0)=T(\varphi,\psi)$, it suffices to show that $(\varphi_0, \psi_0)=(\varphi, \psi)$.
	By the definition of $S$,
	$\varphi=\varphi_0$ in $P_1$ and $\psi=\psi_0$ in $Q_1$.
	Hence, $\varphi_0-\varphi\in H_0^1(\R^N\setminus P_1)$.
	Set $\xi=\varphi_0-\varphi$ in \eqref{eq 4.2}, we obtain 
	\[
	\langle L_r(\varphi_0-\varphi) ,\varphi_0-\varphi\rangle =0.
	\]
	That is 
	\[
	\int_{\R^N\setminus P_1} |\nabla \xi|^2+(K_1 -(p-1) U_r^{p-2})\xi^2 
	=0.\]
Since $\sup_{\R^N\setminus P_1}U_r^{p-2}\to 0$  as $\ell\to 0$, we obtain
$\varphi_0-\varphi=\xi=0$ for large $\ell$. Similarly, $\psi_0-\psi=0$.
Then we have the conclusion.
\end{proof}
From Lemma \ref{lemma4.3} and Lemma \ref{lemma4.1}, for each $(r,\rho)\in \mathbb D_1\times \mathbb D_2$, there is 
a critical point $(\varphi_0(r,\rho),\psi_0(r,\rho))$ of $J_n$ on $\mathbb E=\mathbb E_{r,\rho}$. 
Precisely, we have the following result.
\begin{pro}
	For any $(r,\rho)\in \mathbb D_1\times \mathbb D_2$, there is  	$(\varphi_0(r,\rho),\psi_0(r,\rho))\in \Lambda_\ell$ such that 
\begin{gather*}
	\nabla_{\mathbb E }  J_n(\varphi_0(r,\rho),\psi_0(r,\rho))=0 \mbox{ in } \mathbb E_{r,\rho}.
\end{gather*} Moreover, it holds that
\begin{gather*}
	(\varphi_0(r,\rho),\psi_0(r,\rho))\in C^1(\mathbb D_1\times \mathbb D_2, \mathbb H_s).
\end{gather*} 
\end{pro}
\begin{proof}  
Lemma \ref{lemma4.1} and Lemma \ref{lemma4.3} imply the existence and well-definedness of 
$(\varphi_0(r,\rho),\psi_0(r,\rho))$
such that \[\nabla_{\mathbb E }  J_n(\varphi_0(r,\rho),\psi_0(r,\rho))=0.\]

For each $(r,\rho)$,  we introduce a bounded linear operator $\mathscr{A}_{r,\rho}:\mathbb E_{r_0,\rho_0}\to \mathbb E_{r,\rho}$ defined as follows:
\[
\mathscr{A}_{r,\rho}	(\varphi,\psi)=\left(\varphi-X(r)\frac{\int_{\R^N}X(r)\varphi}{\int_{\R^N}X(r)^2}, \psi-Y(\rho)\frac{\int_{\R^N}Y(\rho)\psi}{\int_{\R^N}Y(\rho)^2}\right)\in\mathbb E_{r,\rho},  (\varphi,\psi) \in \mathbb E_{r_0,\rho_0}.
\]
It is important to note that $\mathscr{A}_{r,\rho}$ is invertible, and specifically, $\mathscr{A}_{r_0,\rho_0}: \mathbb{E}_{r_0,\rho_0} \rightarrow \mathbb E_{r_0,\rho_0}$ is the identity operator.
Given  $(\varphi, \psi, r, \rho)\in \mathbb E_{r_0,\rho_0}\times \mathbb D_1\times \mathbb D_2$, we define the functional
$\mathscr{F}:\mathbb E_{r_0,\rho_0}\times \mathbb D_1\times \mathbb D_2\to \mathbb E_{r_0,\rho_0}$ as
 \[
\mathscr{F}(\varphi,\psi, r,\rho)= \mathscr{A}_{r,\rho}^{-1} \nabla_{\mathbb E} J_n(\mathscr{A}_{r,\rho}(\varphi,\psi)).	
 \]
For the sake of notational convenience, we denote $(\varphi_0, \psi_0)=(\varphi_0(r_0,\rho_0),\psi_0(r_0,\rho_0))$.

 We claim  that 
 \[
	L_n:=\frac{\partial \mathscr{F}(\varphi_0,\psi_0, r_0,\rho_0)}{\partial (\varphi,\psi)}=\frac{\partial^2 J_n(\varphi_0, \psi_0)}{\partial (\varphi,\psi)^2} :\mathbb E_{r_0,\rho_0}\to\mathbb E_{r_0,\rho_0} 
 \mbox{ has a bounded inverse.}
 \]
 If this claim holds true,  the implicit function theorem guarantees the existence of
  a neighborhood of $(\varphi_0,\psi_0, r_0,\rho_0)$,
  containing  a unique tuple $(\bar \varphi(r,\rho), \bar\psi(r,\rho), r,\rho)$
 such that $(\bar \varphi(r,\rho), \bar\psi(r,\rho))$ is $C^1$ with respect to $(r,\rho)$, $(\bar \varphi(r_0,\rho_0), \bar\psi(r_0,\rho_0))=(  \varphi_0,  \psi_0)$, and satisfies
 \[\nabla_{\mathbb E} J_n(\mathscr{A}_{r,\rho}(\bar \varphi(r,\rho), \bar\psi(r,\rho)))=\mathscr{A}_{r,\rho}\mathscr{F}(\bar \varphi(r,\rho), \bar\psi(r,\rho), r,\rho)=0.\]
 By uniqueness, we deduce that $(\varphi_0(r,\rho), \psi_0(r,\rho))=\mathscr{A}_{r,\rho}(\bar \varphi(r,\rho), \bar\psi(r,\rho) )$,  thereby confirming that $(\varphi_0(r,\rho), \psi_0(r,\rho))$ is $C^1$ with respect to $(r,\rho)$.

It remains to verify the aforementioned claim.
To accomplish this, we configure $(u,v)=(\varphi_0+U_{r_0}, \psi_0+V_{\rho_0})$ and focus on establishing the invertibility of the linear operator $L_n$, which is characterized by a bilinear form defined for any pair
  $(\varphi,\psi), (\xi,\eta)\in \mathbb E_{r_0,\rho_0}$,
\begin{equation}\label{Hessian2}
	\begin{aligned}
	\left\langle L_n (\varphi,\psi) ,(\xi,\eta)\right\rangle_\ell  &  =\int_{\mathbb{R}^N}\bigl(\nabla \varphi\nabla\xi+K_1 \varphi\xi-(p-1)\mu  |u|^{p-2} \varphi\xi\bigr)  \\
	&+\int_{\mathbb{R}^N}\left(\nabla \psi\nabla\eta+ K_2  \psi\eta-(p-1)\nu  |v|^{p-2} \psi\eta\right) \\
	&- \beta \int_{\mathbb{R}^N}\left( \partial_{11}G_n (u,v)  \varphi\xi + 
	 \partial_{22}G_n(u,v)   \psi\eta\right)\\
	&- \beta\int_{\R^N} \partial_{12} G_n(u, v) (\psi\xi + \varphi\eta).
	\end{aligned}\end{equation}

We take $\chi_P, \chi_Q$ such that 
$\chi_P=1$ in $P_1$ and $\chi_Q=1$ in $Q_1$, 
$\chi_P=0$ in $P_2$ and $\chi_Q=0$ in $Q_2$,
$0\leq \chi_P, \chi_Q\leq 1$ and $|\nabla \chi_P|, |\nabla\chi_Q|\leq 2/\ln\ln\ell$ in $\R^N$.
Denote 
\[(\varphi_1, \psi_1)=(\chi_P\varphi, \chi_Q\psi)\in    \mathbb E_{0,r_0,\rho_0} :=H_0^1(P_2) \times H_0^1(Q_2)\cap \mathbb E_{r_0,\rho_0}.\]
$(\varphi_2,\psi_2)=(\varphi, \psi)-(\varphi_1,\psi_1)$.
Then we have 
\begin{gather*}
	\|(\varphi,\psi)\|\leq \|(\varphi_1,\psi_1)\|+\|(\varphi_2,\psi_2)\|\leq 2\|(\varphi,\psi)\|,\\
\|L_n(\varphi,\psi)\|\geq 	\|L_n(\varphi_1,\psi_1)\|-\|L_n(\varphi_2,\psi_2)\|.
\end{gather*}
First noting by (Gn2) and $\|(u,v)\|_\infty<\alpha_0$,
\begin{gather*}
	-\beta\left( \partial_{11} G_n(u, v)  \varphi\varphi_2 + 
	\partial_{22}G_n(u,v)   \psi\psi_2\right)	\geq 0,
\end{gather*} 
and by (Gn3), for $\ell$ large
\begin{gather*}
	\partial_{12}G_n(u,v)=O(|uv|^\sigma)=O(e^{-\ell})\mbox{ in } \R^N, \\
	(p-1)\mu  |u|^{p-2}<\frac12 \mbox{ in } \R^N\setminus P_2,\\
	(p-1)\nu  |v|^{p-2} <\frac12 \mbox{ in } \R^N\setminus Q_2,	
\end{gather*}
we have 
\begin{equation}\label{eq4.7}
	\begin{aligned}
		&\|L_n(\varphi,\psi)\|\|(\varphi_2,\psi_2)\|\geq \left\langle L_n (\varphi,\psi) ,(\varphi_2,\psi_2)\right\rangle \\
		\geq &\int_{\R^N} \nabla \varphi\nabla\varphi_2 +\nabla\psi\nabla \psi_2 +\frac12\int_{\R^N}( \varphi_2^2+ \psi_2^2)-O(e^{-\ell})\|(\varphi,\psi)\|^2\\
		\geq &\frac12 \|(\varphi_2,\psi_2)\|^2-O((\ln\ln\ell)^{-1}) \|(\varphi,\psi)\|^2.
	\end{aligned}
\end{equation}

On the other hand, 
  for any $(\xi,\eta)\in  \mathbb E_{r_0,\rho_0}$, we have 
\begin{align*}
	&\left\langle L_n (\varphi,\psi) ,(\xi_1,\eta_1)\right\rangle 
	     \\
	=&\left\langle L_n (\varphi_1,\psi_1) ,( \xi, \eta)\right\rangle  
	 +\int_{\R^N}\nabla\varphi\nabla\chi_P \xi -\nabla\chi_P\nabla\xi \varphi 
	 +\int_{\R^N}\nabla\psi\nabla\chi_Q \eta -\nabla\chi_Q\nabla\eta \psi\\
	 =&\left\langle L_n (\varphi_1,\psi_1) ,( \xi, \eta)\right\rangle +O((\ln\ln\ell)^{-1})\|(\varphi,\psi)\|\|(\xi,\eta)\|,
\end{align*}
where we  write $(\xi,\eta)=(\xi_1,\eta_1)+(\xi_2,\eta_2)$  with 
$(\xi_1, \eta_1)=(\chi_P\xi, \chi_Q\eta)$.
Since as $\ell\to +\infty$
\begin{align*}
	&|u|^{p-2}-U_{r_0}^{p-2}\to 0,\quad  \partial_{11} G_n(u, v) \to 0 \mbox{ in } P_2,\\
	&|v|^{p-2}-V_{\rho_0}^{p-2}\to 0, \quad  \partial_{22} G_n(u, v) \to 0 \mbox{ in } Q_2,
\end{align*}
we have 
\[\begin{aligned}
	\left\langle L_n (\varphi_1,\psi_1) ,(\xi,\eta)\right\rangle &=\left\langle L_{r_0} \varphi_1,\xi\right\rangle+\left\langle L_{\rho_0} \psi_1,\eta\right\rangle \\
	&+(p-1)\int_{\R^N}\mu(|u|^{p-2}-U_{r_0}^{p-2})\varphi_1\xi +\nu(|v|^{p-2}-V_{\rho_0}^{p-2}) \psi_1\eta \\
	&- \beta \int_{\mathbb{R}^N}\left( \partial_{11} G_n(u, v) \varphi_1\xi + 
	 \partial_{22} G_n(u, v)   \psi_1\eta\right)\\
	&- \beta\int_{\R^N}  \partial_{12} G_n(u, v) (\psi_1\xi + \varphi_1\eta)\\
	=&\left\langle L_{r_0} \varphi_1,\xi\right\rangle+\left\langle L_{\rho_0} \psi_1,\eta\right\rangle
	+o(1)\|(\varphi,\psi)\|\|(\xi,\eta)\|.
\end{aligned}	
\]	
As a result of Lemma \ref{lemma2.2}, we have 
\begin{align*}
	\|L_n (\varphi,\psi)\| \geq& \sup\Set{\left\langle L_n (\varphi,\psi) ,(\xi,\eta)\right\rangle | 
	\|(\xi,\eta)\|\leq 1, (\xi,\eta)\in  \mathbb E_{0,r_0,\rho_0}}\\
	\geq  &\sup\Set{\left\langle L_n (\varphi,\psi) ,( \xi_1, \eta_1)\right\rangle | 
	\|(\xi,\eta)\|\leq \frac12, (\xi,\eta)\in  \mathbb E_{r_0,\rho_0}}\\
	\geq&  \varrho \|(\varphi_1,\psi_1)\|+o(1)\|(\varphi,\psi)\|.
\end{align*}
 Together with \eqref{eq4.7}, we have 
 \begin{align*}
	\|L_n (\varphi,\psi)\| \|(\varphi,\psi)\| \geq& \frac12\|L_n (\varphi,\psi)\| (\|(\varphi_1,\psi_1)\| + \|(\varphi_2,\psi_2)\|)\\
	\geq &\frac12\varrho\|(\varphi_1,\psi_1)\|^2 +\frac14 \|(\varphi_2,\psi_2)\|^2 +o(1)\|(\varphi,\psi)\|^2\\
	\geq &\frac18\min\set{\varrho, 1}\|(\varphi,\psi)\|^2.
 \end{align*}
This completes the proof.
\end{proof}

\section{Proof of the main results}
\subsection{Completion of the proof of Theorem \ref{thm1} and Theorem \ref{thm1'}}
Denoting 
\begin{align*}
I_\mu(u)=\frac{1}{2}\int_{\mathbb{R}^N}\big(|\nabla u|^2+K_1(|x|)u^2\big)    
	 -\frac{\mu}{ p}\int_{\mathbb{R}^N} |u|^{p}  \\
	 I_\nu(v)=\frac{1}{2}\int_{\mathbb{R}^N}\big(|\nabla v|^2+K_2(|x|)v^2\big)    
	 -\frac{\nu}{ p}\int_{\mathbb{R}^N} |v|^{p},
\end{align*}
we have 
 \begin{align*}
&I_n(U_r+\varphi_0(r,\rho), V_\rho+\psi_0(r,\rho)) \\
=&I_\mu(U_r+\varphi_0(r,\rho))+I_\nu(V_\rho+\psi_0(r,\rho))+O(e^{-\ell})	\\
=&I_\mu(U_r )+I_\nu(V_\rho ) +O(\frac{1}{r^{m+\tau}}+\frac1{\rho^{m+\tau}})\\
=&\ell(A +\frac{B_1}{r^m}+\frac{B_2}{\rho^m}-C_1(\frac{\ell}{r})^\frac{N-1}{2}e^{-\frac{ 2\pi r}{\ell}}-C_2(\frac{\ell}{\rho})^\frac{N-1}{2}e^{-\frac{ \pi \rho}{\ell}}+O(\frac{1}{r^{m+\tau}}+\frac1{\rho^{m+\tau}})).
 \end{align*}
 We show that this function 
 of two variables attains a local maximum at some point  $(r, \rho)\in \mathbb D_1\times \mathbb D_2$.
Set 
\[
s=\frac{r}{\ell\ln\ell}-\frac{m}{2\pi},\quad t =\frac{\rho}{\ell\ln\ell}-\frac{m}{\pi}.
\]
Then 
\[ I_n(U_r+\varphi_0(r,\rho), V_\rho+\psi_0(r,\rho))=
\ell A +\ell^{1-m} (\ln \ell)^{-m} F(s,t)
\]
where 
\[ F(s, t)=
 f(s) +h(t)
+ O\left( \ell^{-\tau} (\ln \ell)^{-\tau} \right) 
\]
with
\[
f(s)=B_1 \left( s + \frac{m}{2\pi} \right)^{-m} - C_1 (\ln \ell)^{m - \frac{N-1}{2}} e^{-2\pi s\ln\ell} \left( s + \frac{m}{2\pi} \right)^{-\frac{N-1}{2}},
\]
\[
h(t)=B_2 \left( t + \frac{m}{\pi} \right)^{-m} -C_2 (\ln \ell)^{m - \frac{N-1}{2}} e^{-\pi t\ln\ell} \left( t + \frac{m}{\pi} \right)^{-\frac{N-1}{2}}.
\]
One can check that for $|s|\leq \frac{m}{4\pi}$,
\[
f'(s)=-mB_{1}\left(s + \dfrac{m}{2\pi}\right)^{-m-1} 
+ C_{1}(\ln \ell)^{m - \frac{N+1}{2}}e^{-2\pi s\ln\ell}
 [ 2\pi  \left(s + \dfrac{m}{2\pi}\right)^{-\frac{N-1}{2}} 
  + O(\frac{1}{\ln\ell})].
\]
\[
f''(s)=m(m+1)B_{1}\left(s + \dfrac{m}{2\pi}\right)^{-m-2}-C_{1}(\ln \ell)^{m - \frac{N+3}{2}}e^{-2\pi s\ln\ell}
 [
   4\pi^{2}\left(s + \dfrac{m}{2\pi}\right)^{-\frac{N-1}{2}} +O(\frac{1}{\ln \ell })].
\]
Therefore, $f$ has a critical point $s_\ell$
with 
\[
s_\ell =\frac{(m-\frac{N+1}{2})\ln\ln\ell +O(1)}{2\pi \ln \ell}.
\]
and   for some $\epsilon>0$,
\[
f''(s)\leq -2\epsilon\ln\ell  \mbox{ for } |s-s_\ell|\leq \frac{1}{\ln\ell}.
\]
Therefore,
\[
f(s)\leq f(s_\ell)- \epsilon(\ln\ell)(s-s_\ell)^2,\quad  |s-s_\ell|\leq \frac{1}{\ln\ell}.
\]
Similarly, decreasing $\epsilon$ if necessary,
$h$ has a critical point $t_\ell$ and 
\[
h(t)\leq h(t_\ell)- \epsilon(\ln\ell)(t-t_\ell)^2,\quad |t-t_\ell|\leq \frac{1}{\ln\ell}.
\]
Then if $|s-s_\ell|= \ell^{-\frac{\tau}{2}}$, or $|t-t_\ell|=\ell^{-\frac{\tau}{2}}$,
\[
F(s,t)\leq f(s_\ell)+h(t_\ell)- \epsilon\ell^{-\tau}\ln\ell +O\left( \ell^{-\tau} (\ln \ell)^{-\tau} \right)<F(s_\ell, t_\ell). 
\]
There $F(s,t)$ has a local maximum in $(-\ell^{-\frac\tau2}, \ell^{-\frac\tau2})^2$.
 Consequently, for each sufficiently large $\ell$, the perturbed problem \eqref{eq1.1n} admits a solution if $n\geq n_\ell$. Taking the limit as $n \to +\infty$,
 it converges to a solution $(u_\ell, v_\ell)$ to \eqref{eq0}.
 To see $u_\ell, v_\ell$ are nonnegative, we employ $(u_\ell^-, v_\ell^-)$ as test functions in equation \eqref{eq0},
 yielding
 \[
\|(u_\ell^-, v_\ell^-)\|\leq \mu\|u_\ell^-\|_{L^p}^p+\nu\|v_\ell^-\|_{L^p}^p.
 \]
 Given that
 $(u_\ell, v_\ell)\in \Lambda_\ell$,  this inequality leads us to   $(u_\ell^-, v_\ell^-)=0$,
 concluding the proof.  
 \hfill \qedsymbol
\subsection{The phenomenon of dead cores}
We will demonstrate that the segregated solutions obtained are dead core solutions. To obtain this conclusion, we mainly rely on the following lemma, which can be deduced from \cite[Theorem~8.4.2–8.4.7]{Pucciserrin}; for self-containedness, we also provide a more direct proof in Appendix~\ref{appendix}.

 \begin{lem}\label{lem5.1}
Let $\tau > 0$ be fixed as in \eqref{tau} and let $\sigma' \in (0,1)$ be given by assumption (G4). 	There is $c_\tau>0$ such that 
	\begin{equation*}
		\begin{cases}
			-\Delta w+ w^{\sigma'}=0  &\mbox{ in } B_1(0)\\
			w=c_\tau \quad &\mbox{ on } \partial B_1(0),
		\end{cases}
	\end{equation*}
	has a unique distribution nonnegative radial solution $w(x)=w(|x|)$,
	satisfying
	 $w \in C^1(B_1(0))$, $w=0$ in $B_{(m+\tau)/(m+\frac32\tau)}(0)$ and  $w'(r)>0$ if $w>0$.
 \end{lem}
 
\begin{proof}[Proof of Theorem \ref{deadcore}]
For $j=1,\dots, 2\ell$, on $\partial  B_{\frac{2m+3\tau}4\ln \ell}(y^j)$, we have 
\[
u_\ell =O(e^{-\ell}), \quad v_\ell \geq C(\ln\ell)^{-\frac{N-1}{2}} \ell^{-\frac{2m+3\tau}4}-\ell^{-\frac{m}2-\tau}.
\]
By (G4), we assume there is $A>0$ such that 
in $B_{\frac{2m+3\tau}4\ln \ell}(y^j)$,there holds
\[
\partial_1 G(u_\ell, v_\ell)\geq A u_\ell^{\sigma'} v_\ell^{\sigma'+1}.
\]
From equation \eqref{eq0'}, it follows that
\[
-\Delta u_\ell 	+ 16\ell^{-\frac{m}2} u_\ell^{ \sigma'}\leq -\Delta u_\ell 	+ (K_1 u_\ell^{1-\sigma'}-\mu u_\ell^{p-\sigma'}-\beta v_\ell^{\sigma'+1})u_\ell^{\sigma'}\leq 0.
\]
Consider 
\[
w_\ell(x)= \ell^{-\frac{m}{2(1-\sigma')}}	\left(  {(2m+3\tau)\ln \ell} \right)^{\frac{2}{1-\sigma'}} w\left(\frac{4x}{(2m+3\tau)\ln\ell}\right).
\]
It can be verified that
\begin{equation*}
	\begin{cases}
		-\Delta w_\ell+16 \ell^{-\frac{m}{2}}w_\ell^{\sigma'}=0  &\mbox{ in } B_{{\frac{2m+3\tau}4}\ln \ell}(0)\\
		w_\ell >c_\tau \ell^{-\frac{m}{2(1-\sigma')}}	\left(  {(2m+3\tau)\ln \ell} \right)^{\frac{2}{1-\sigma'}}  \quad &\mbox{ on } \partial B_{{\frac{2m+3\tau}4}\ln \ell}(0)\\
		w_\ell=0 &\mbox{ in } B_{{\frac{m+\tau}2}\ln \ell}(0).
	\end{cases}
\end{equation*}
Therefore, 
$w_\ell(\cdot - y^j)$ is a supersolution of $-\Delta u +\ell^{-\frac{m}{2}}u^{\sigma'}=0$ and 
satisfies $w_\ell(\cdot - y^j)>u_\ell $ on $\partial B_{{\frac{2m+3\tau}4}\ln \ell}(y^j)$.
Since $u_\ell$ is nonnegative as shown in Section 5.1, we must have
\[
	u_\ell=0\mbox{ in } \bigcup_{j=1}^{2\ell} B_{\frac{m+\tau}2\ln \ell}(y^j).	
\]
Note that $\frac{|y^1-y^2|}2=\rho_\ell\sin\frac{\pi}{2\ell}=(\frac{m}{2}+o(1))\ln \ell<\frac{m+\tau}2\ln \ell$.
We know that $y^j$,$j=1,\dots,2\ell$ contains in a connected component of $\set{x | u_\ell(x) =0}$.
If $N=2$, then $\tilde u_\ell=u_{\ell}|_{\R^N\setminus B_{\rho_\ell-\sqrt{m\tau/2}\ln\ell}}\in H_0^1(\R^N\setminus B_{\rho_\ell-\sqrt{m\tau/2}\ln\ell}(0))$.
Testing $-\Delta u+K_1 u= u^{p-1} +\beta \partial_1G(u,v)$ by $\tilde u_\ell$,
we have 
\[
\|\tilde u_\ell\|^2\leq \int_{\R^N} |\tilde u_\ell|^p\leq C\|\tilde u_\ell\|^p.
\]
Since $\|\tilde u_\ell\|\to 0$ as $\ell\to +\infty$, we can conclude 
that $\tilde u_\ell=0$.
Similar arguments for $v_\ell$ completes the proof.
\end{proof}

 \medskip

 \appendix
\section{Proof of Lemma \ref{lem5.1}}\label{appendix}

	\begin{proof}
\textbf{Step 1. } We show that 
for each $c>0$ the equation 
\begin{equation*}
		\begin{cases}
			-\Delta w+ w^{\sigma'}=0  &\mbox{ in } B_1(0)\\
			w=c \quad &\mbox{ on } \partial B_1(0),
		\end{cases}
	\end{equation*}
	has a unique nonnegative solution $w$.. Moreover, this solution is radially symmetric and satisfies $w'(r)>0$ if $w(r)>0$.
To prove this, consider the energy functional
\[
\mathscr F(u)=\frac12\int_{B_1(0)}|\nabla u|^2 + |u|^{\sigma'+1}
\]
defined on the affine space
\[
M=\set{u\in H^1(B_1(0)) | u=c \text{ on } \partial B_1(0)}.
\]
We know $\mathscr F$ is strictly convex and coercive  on $M$.
Hence, it has a unique minimizer $w$ which is the unique solution.
Since $\mathscr F(|w|)\leq \mathscr F(w)$ and  $|w|\in M$,  by uniqueness, $w=|w|\geq 0$. Moreover, it is easy to check that $\mathscr F$ has 
a   minimizer on $M\cap H_{r}^1(B_0(1))$, which is a nonnegative radial
solution. By uniqueness again, we know $w$ is radially symmetric.
In radial coordinates, the equation   reduces to
	\[
	(r^{N-1}w')'(r)=r^{N-1}w^{\sigma'}(r).
	\]
 Therefore, $r^{N-1}w'$ is increasing and  
 $w'(r)>0$ wherever $w(r)>0$.

	\textbf{Step 2.}
	For each $a\in(0,1)$,
	we find $c_a$ such that $w=0$ in $B_a(0)$.  Let $q> \tfrac{2}{1-\sigma'} $ to be fixed later. Define
	\[
	\overline w(x):=\big[(|x|-a)_+\big]^q,\qquad c:=(1-a)^q>0.
	\]
	Then $\overline w$ is radial, $\overline w\equiv0$ in $B_{a}$, $\overline w=c$ on $\partial B_1$,
	and for $r=|x|>a$,
	\[
	\overline w'(r)=q(r-a)^{q-1}>0,\qquad \overline w''(r)=q(q-1)(r-a)^{q-2}.
	\]
	Since $q> \tfrac{2}{1-\sigma'}>2$, we have $\overline w'(a)=\overline w''(a)=0$, hence $\overline w\in C^2(B_1(0))$.
	A direct computation gives, for $r>a$,
	\[
	\Delta\overline w=\frac{q(N-1)}{r}(r-a)^{q-1}+q(q-1)(r-a)^{q-2},
	\]
	whence
	\[
	-\Delta\overline w+\overline w^{\sigma'}=
	-\Big(\tfrac{q(N-1)}{r}(r-a)^{q-1}+q(q-1)(r-a)^{q-2}\Big)
	+(r-a)^{q\sigma'}.
	\]
	Factoring $(r-a)^{q-2}$ and using $\frac{r-a}{r}\le1-a$,
	\[
	-\Delta\overline w+\overline w^{\sigma'}
	\ge \Big[(1-a)^{\,q\sigma'+2-q}-q(N-1)(1-a)-q(q-1)\Big]\,(r-a)^{q-2}.
	\]
	Because $q\sigma'+2-q\to -\infty$, as  $q\to +\infty$, so by fixing $q$ sufficiently large, the bracket
	is nonnegative, and thus
	\[
	-\Delta\overline w+\overline w^{\sigma'}\ge0\quad\text{in }B_1(0).
	\]
	Therefore $\overline w$ is a global supersolution with $\overline w=0$ in $B_a(0)$.
Now since $w$ is a solution, we have 
\[
-\Delta(w-\overline w)+w^{\sigma'}-\overline w^{\sigma'}\leq 0\text{ in } B_1(0).
\]
Testing the equation by $ (w-\overline w)^+\in H_0^1(B_1(0))$, we obtain $\|\nabla(w-\overline w)^+\|_2^2\leq 0$.
Hene $w\leq \overline w$.
	\end{proof}

\noindent\textbf{Conflict of Interest}:
The authors declare that they have no conflict of interest.

 \medskip
\noindent\textbf{Data Availability}:
 	No   data were generated or analysed in this study.
 	
\medskip
\noindent\textbf{Acknowledgments}: The authors are deeply grateful to Professor Angela Pistoia for her valuable discussions and insightful suggestions throughout the preparation of this paper. Her expertise and guidance have been instrumental in enhancing the quality of this work.

This research was supported by the National Natural Science Foundation of China (Grant Nos. 12271539, 12371107) and the China Scholarship Council (CSC) during the first author's visit to Sapienza University of Rome.


\begin{thebibliography}{99}
 
	\bibitem{AA} N. Akhmediev, A. Ankiewicz,   {Partially coherent solitons on a finite background,} Phys. Rev. Lett, 82 (1999) 2661-2664.
	\bibitem{3}A. Ambrosetti, E. Colorado,
	 {Standing waves of some coupled nonlinear Schr\"odinger equations}, J. Lond. Math. Soc.,
   75 (2007) 67-82.
   \bibitem{ambrosetti-colorado-ruiz}
A. Ambrosetti, E. Colorado, D. Ruiz,
  {Multi-bump solitons to linearly coupled systems
of Nonlinear Schr\"odinger equations}, Calc. Var. Partial Differential Equations,
30(1) (2007) 85-112.

\bibitem{BC13}
W. Bao, Y. Cai,
  {Mathmatical theory and numerical methods for Bose-Einstein condensation.}  Kinetic and Related Models, {6} (2013) 1-135.

   %\bibitem{8} T. Bartsch, Z. Wang, J. Wei,  {Bound states for a coupled Schr\"odinger system}, J. Fixed Point Theory Appl., 2(2007) 353-367.
% %\bibitem{BL1}Berestycki, H, Lions, P.L.:
%  Nonlinear scalar field equations, I existence of a ground state.
%  Arch. Ration. Mech. Anal. 82(4) (1983) 313-345 .

 %\bibitem{BL2}Berestycki, H, Lions, P.L.: Nonlinear scalar field equations, II existence of infinitely many solutions.Arch. Ration. Mech. Anal. 82(4), 347--375, (1983).

\bibitem{BSSS} H. Buljan, T. Schwartz, M. Segev, M. Soljacic, D. Christoudoulides, Polychromatic partially spatially incoherent solitons in a noninstantaneous Kerr nonlinear medium, J. Opt. Soc. Amer. B 21 (2004) 397--404.
\bibitem{BBT}J.L. Bona, D.K. Bose, R.E.L. Turner, Finite-amplitude steady waves in stratified fluids, J. Math. Pures Appl. 62
(1983) 389--439.

\bibitem{BL} H. Brezis, E.H. Lieb, Minimum action solutions of some vector field equations, Comm. Math. Phys. 96 (1984)
97--113.
  
\bibitem{BT14}
J. Byeon, K. Tanaka,    Semiclassical standing waves with clustering peaks for nonlinear Schrödinger equations. Mem. Amer. Math. Soc. 229, viii+89 (2014).

\bibitem{CZ} R. Cipolatti, W. Zumpichiatti, On the existence regularity of ground state for a nonlinear system of coupled
Schr\"odinger equations in RN, Comput. Appl. Math. 18 (1999) 15--29.

\bibitem{Clapp25}
M. Clapp, V. Hernández-Santamaría, A. Saldaña,  
A strong unique continuation property for weakly coupled elliptic systems.
J. Math. Anal. Appl. 544 (2025), no. 2, Paper No. 129069, 13 pp.

\bibitem{Coti92}
V. Coti Zelati, P. H. Rabinowitz,   Homoclinic type solutions for a semilinear elliptic PDE on $\mathbb R^n$. Comm. Pure Appl. Math. 45, 1217--1269 (1992).

\bibitem{13} N. Dancer, J. Wei, T. Weth,  {A priori bounds versus multiple existence of positive solutions for a nonlinear Schr\"odinger system}, Ann. Inst. H. Poincar\'e Anal. Non Lin\'eaire, 27 (2010) 953-969.
 
 \bibitem{15} D. G. de Figueiredo, O. Lopes,  {Solitary waves for some nonlinear Schr\"odinger systems}, Ann. Inst. H. Poincar\'e Anal. Non Lin\'eaire, 25 (2008) 149-161.
 

\bibitem{guo1} Q. Guo, J.Yang, Excited states for two-component Bose-Einstein condensates in dimension two, Journal of Differential Equations,343 (2023) 659-686.

\bibitem{guo2} Q. Guo, L. Zhao, Non-degeneracy of synchronized vector solutions for weakly coupled nonlinear Schr\"odinger systems, Proceedings of the Edinburgh Mathematical Society, 65 (2022) 441--45.

%\bibitem{guo3}Q. Guo, H. Xie, Existence and local uniqueness of normalized solutions for two-component Bose-Einstein condensates, Zeitschrift fur Angewandte Mathematik und Physik,2021, 72(6):189.



\bibitem{GLWZ19-1}
Y. Guo, S. Li, J. Wei, X. Zeng,
 {Ground states of two-component attractive
Bose-Einstein condensates I: Existence and uniqueness.} J. Funct. Anal., 276 (2019) 183-230.

 \bibitem{GLWZ19-2}
Y. Guo, S. Li, J. Wei, X. Zeng,
 {Ground states of two-component attractive Bose-Einstein condensates II: Semi-trivial limit behavior.} Trans. Amer. Math. Soc., 371 (2019) 6903-6948.
\bibitem{GLW}
Y. Guo, C. Lin, J. Wei,
 {Local uniqueness and refined spike profiles of ground states for two-dimensional attractive Bose-Einstein condensates.} SIAM J. Math. Anal., 49 (2017) 3671-3715.


\bibitem{GS}
Y. Guo, R. Seiringer,
 {On the mass concentration for Bose-Einstein condensates with attractive interactions.} Lett. Math. Phys., 104 (2014) 141-156.


\bibitem{Guo1}
Y. Guo, X. Zeng, H. Zhou,  {Energy estimates and symmetry breaking
in attractive Bose-Einstein condensates with ring-shaped potentials.}  Ann. Inst. H.
Poincar\'e Anal. Non Lin\'eaire, 33 (2016) 809-828.

\bibitem{glwjde}Y. Guo, B. Li, J. Wei, Entire nonradial solutions for non-cooperative coupled
elliptic system with critical exponents in $\R^3$,J. Differential Equations 256 (2014) 3463--3495.




 %\bibitem{Hioe}F.T. Hioe, Solitary waves for N coupled nonlinear Schr\"odinger equations, Phys. Rev. Lett. 82 (1999) 1152--1155.
\bibitem{HMEWC}D. S. Hall, M. R. Matthews,  J. R. Ensher, C. E. Wieman, E. A. Cornell, Dynamics of component separation in a binary mixture of Bose-Einstein conden-
 sates. Phys. Rev. Lett. 81, 8 (1998)  1539--1542.

 \bibitem{KTU} K. Kasamatsu, M. Tsubota,  M. Ueda,  Vortices in multicomponent Bose-
Einstein condensates. Int. J. Mod. Phys. B 19  (2005) 1835.


 \bibitem{21}Z. Lin, Z. Wang,  {Multiple bound states of nonlinear Schr\"odinger systems}, Comm. Math. Phys., 282 (2008) 721-731.
 \bibitem{LW05}
 T. Lin, J. Wei,  {Ground state of $N$ coupled nonlinear Schr\"odinger equations in $\R^N$, $N\leq 3$.} Comm. Math. Phys., 255 (2005) 629-653.
 \bibitem{lw} T. Lin, J. Wei,  {Spikes in two-component systems of nonlinear Schr\"odinger equations with trapping potentials}, J. Differential Equations, 229 (2006) 538-569.

%\bibitem{lw05}T. Lin, J. Wei, Spikes in two coupled nonlinear Schr\"odinger equations, Ann. Inst. H. Poincar\'e Anal. Non Lin\'eaire 22 (2005) 403-439.


 %\bibitem{LP14} W. Long, S. Peng,  { Segregated vector solutions for a class of Bose-Einstein systems}.  J. Differential Equations, 257 (2014) 207-230.

 %\bibitem{MZ} L. Ma, L. Zhao, Uniqueness of some coupled nonlinear Schr\"odinger systems and their applications, J. Differential Equations, 245, 2551-2565, (2008).

%\bibitem{24} E. Montefusco, B. Pellacci, M. Squassina,  {Semiclassical states for weakly coupled nonlinear Schr\"odinger systems}, J. Eur. Math. Soc., 10 (2008) 41-71.

\bibitem{MMP} L. Maia, E. Montefusco, B. Pellacci, Positive solutions for a weakly coupled nonlinear Schr\"odinger systems, 
J. Differential Equations, 229 (2006) 743-767.

%\bibitem{Mandel}R. Mandel, Uniqueness results for semilinear elliptic systems on $R^n$, Math. Nachr. 287(16), 1828--1836, (2014)
%\bibitem{26}B. Noris, H. Tavares, S. Terracini, G. Verzini,  {Uniform H\"older bounds for nonlinear Schr\"odinger systems with strong competition}, Comm. Pure Appl. Math., 63 (2010) 267-302.


\bibitem{pistoia-coron}A. Pistoia, N. Soave, On Coron's problem for weakly coupled elliptic systems, Proc. London Math. Soc. (3) 116 (2018) 33--67.

\bibitem{pistoiavaira}
A. Pistoia, G. Vaira, Segregated solutions for nonlinear Schr\"odinger systems with weak interspecies forces, Comm. Partial Differential Equations   47  (2022)  2146--2179.

%\bibitem{Palais} Palais, R.S.: The principle of symmetric criticality. Comm. Math. Phys. 69, 19--30, (1979).
 \bibitem{PW13}
 S. Peng, Z.-Q. Wang,  
 Segregated and synchronized vector solutions
for nonlinear Schr\"odinger systems. 
Arch. Rational Mech. Anal., 208 (2013) 305-339.

\bibitem{Pucciserrin}
P. Pucci, J. Serrin,  
The maximum principle.
Progr. Nonlinear Differential Equations Appl., 73
Birkhäuser Verlag, Basel, 2007. x+235 pp.
% \bibitem{29} A. Pomponio,  {Coupled nonlinear Schr\"odinger systems with potentials}, J. Differential Equations, 227 (2006) 258-281.
 \bibitem{RL14}
 J. Royo-Letelier,  {Segregation and
 symmetry breaking of strongly coupled two component
 Bose-Einstein condensates in a harmonic trap}.
 Calc. Var. Partial Differential Equations, 49 (2014) 103-124.

\bibitem{Stuart}C.A. Stuart, Bifurcation in $L^p(\R^N)$ for a semilinear elliptic equation, Proc. London Math. Soc. 57 (1988) 511--541.

%\bibitem{St} Strauss, W.A.: Existence of solitary waves in higher dimensions. Comm. Math. Phys. 55(2), 149--162, (1977).
\bibitem{31}B. Sirakov,  {Least energy solitary waves for a system of nonlinear Schr\"odinger equations in $\R^n$}, Comm. Math. Phys., 271 (2007) 199-221.

\bibitem{32} S. Terracini, G. Verzini,  {Multipulse phase in k-mixtures of Bose-Einstein condensates}, Arch. Ration. Mech. Anal., 194 (2009) 717-741.


%\bibitem{TZ} Tanaka, K., Zhang, C.: Multi-bump solutions for logarithmic Schr\"odinger equations. Calc. Var. Partial Differ. Equ. 56(33), (2017).
  
%\bibitem{lp} W. Long, S. Peng, Segregated vector solutions for a class of Bose--Einstein systems, J. Differential Equations 257 (2014) 207--230.

\bibitem{WZZ} 
Z.-Q. Wang, C. Zhang, Z. Zhang,  {Multiple bump solutions to logarithmic scalar field equations,} Advances in Differential Equations, 28 (2023) 981-1036.
  
\bibitem{WY} J. Wei, S. Yan,
Infinitely many positive solutions for the nonlinear Schr\"odinger equations in $\R^N$.
Calc. Var. Partial Differ. Equ. 37(3--4) (2010) 423-439.


%\bibitem{bdw} T. Bartsch, N. Dancer, Z.-Q. Wang, A Liouville theorem, a-priori bounds, and bifurcating branches of positive solutions for a nonlinear elliptic system, Calc. Var. Partial Diff. Equ., 37, 345-361 (2010).


%\bibitem{lw06}T. C. Lin, J. Wei, Solitary and self-similar solutions of two-component systems of nonlinear Schr\"odinger equations, Physica D: Nonlinear Phenomena,220, 99-115 (2006).

%\bibitem{tv}S. Terracini, G. Verzini, Multipulse phase in k-mixtures of Bose-Einstein condensates, Arch. Ration. Mech. Anal., 194, 717-741, (2009).

%\bibitem{tory}W. C. Troy, Symmetry properties in systems of semilinear elliptic equations. J. Differential Equations 42 (1981), 400-413.

%\bibitem{ww} J. Wei, T. Weth, Nonradial symmetric bound states for s system of two coupled Schr\"odinger equations, Rend. Lincei Mat. Appl., 18, 279-293 (2007).

\bibitem{awy}W. Ao, L. Wang,  W.Yao, Infinitely many solutions for nonlinear Schr\"odinger
system with non-symmetric potentials, Commun. Pure Appl. Anal. 15 no. 3 (2016)  965-989.

\bibitem{aw} W. Ao, J.  Wei, Infinitely many positive solutions for nonlinear equations with non-symmetric potentials, Calc. Var. Partial Differential Equations 51  no. 3-4 (2014) 761-798.

%\bibitem{lty} Wei Long, Zhongwei Tang, and Sudan Yang, Many synchronized vector solutions for a Bose- Einstein system, Proc. Roy. Soc. Edinburgh Sect. A 150 (2020), no. 6, 3293-3320.


%\bibitem{cps} Giovanna Cerami, Donato Passaseo, and Sergio Solimini, Nonlinear scalareld equations: existence of a positive solution with infinitely many bumps, Ann. Inst. H. Poincar\'e Anal. Non Lin\'eaire 32 (2015), no. 1, 23-40.

%\bibitem{pww} Shuangjie Peng, Qingfang Wang, and Zhi-Qiang Wang, On coupled nonlinear Schr\"odinger systems with mixed couplings, Trans. Amer. Math. Soc. 371 (2019), no. 11, 7559-7583.
\bibitem{wz}L. Wang, C. Zhao, Infinitely many solutions for nonlinear Schr\"odinger equations
with slow decaying of potential, Discrete Contin. Dyn. Syst. 37  no. 3 (2017) 1707-1731.

\bibitem{zheng} L. Zheng, Segregated vector solutions for the nonlinear Schr\"odinger systems in $\R^3$,
Mediterr. J. Math. 14  no. 3 (2017) 107.

\end{thebibliography}
\end{document}